\documentclass[11pt,draft]{amsart}

\usepackage{tikz}
\usetikzlibrary{shapes,snakes,calc,arrows}
\usetikzlibrary{decorations.pathreplacing,shapes.geometric}
\usetikzlibrary{calc,positioning}

\usepackage{mathabx}
\usepackage{amsrefs}
\usepackage{amsthm}
\usepackage{amscd}
\usepackage{amsfonts}
\usepackage{amsmath}
\usepackage{amssymb}
\usepackage{mathrsfs}
\usepackage{enumitem}
\usepackage{multirow}
\usepackage{verbatim}
\usepackage{url}
\usepackage{graphicx}
\usepackage{color}

\usepackage{ifthen}
\usepackage{microtype}
\usepackage{hyperref}
\usepackage{wasysym} 
\usepackage{breqn}

\usepackage[T1]{fontenc} 

\newtheorem{theorem}{Theorem}[subsection]
\newtheorem{lemma}[theorem]{Lemma}
\newtheorem{proposition}[theorem]{Proposition}
\newtheorem{prop-and-notation}[theorem]{Proposition and Notation}
\newtheorem{corollary}[theorem]{Corollary}

\theoremstyle{definition}
\newtheorem{definition}[theorem]{Definition}
\newtheorem{definition-and-remark}[theorem]{Definition and Remark}
\newtheorem{notation}[theorem]{Notation}
\newtheorem{notation-and-remark}[theorem]{Notation and Remark}
\newtheorem{remark-and-notation}[theorem]{Remark and Notation}
\newtheorem{remark}[theorem]{Remark}
\newtheorem{outline}[theorem]{Outline}
\newtheorem{example}[theorem]{Example}

\def\greyscale{1}
\ifthenelse{\greyscale>0}{
	
	\definecolor{rline}{rgb}{0 0 0}
}{
	
	\definecolor{rline}{rgb}{1 0 0}
}

\newcommand{\cA}{ {\mathcal A} }
\newcommand{\bC}{ {\mathbb C} }
\newcommand{\cD}{ {\mathcal D} }
\newcommand{\bE}{ {\mathbb E} }
\newcommand{\cI}{ {\mathcal I} }
\newcommand{\bN}{ {\mathbb N} }
\newcommand{\BNC}{ \mathrm{BNC} }
\newcommand{\NC}{ \mathrm{NC} }
\newcommand{\cP}{{\mathcal P}}
\newcommand{\bR}{ {\mathbb R} }
\newcommand{\cR}{{\mathcal R}}
\newcommand{\cS}{{\mathcal S}}
\newcommand{\cT}{{\mathcal T}}
\newcommand{\cTalg}{ {\mathcal T}^{( \mathrm{alg} )} }
\newcommand{\cV}{{\mathcal V}}
\newcommand{\bZ}{ {\mathbb Z} }

\newcommand{\xR}{ Q }
\newcommand{\xlambda}{ \gamma }
\newcommand{\xLambda}{ \Gamma }
\newcommand{\xsigma}{ \theta }
\newcommand{\ecpi}{ \stackrel{\pi}{\sim} }
\newcommand{\ecsigma}{ \stackrel{\sigma}{\sim} }
\newcommand{\xivac}{ \xi_{\mathrm{vac}} }
\newcommand{\phivac}{ \varphi_{\mathrm{vac}} }

\newcommand{\alt}{ \mathrm{alt} }
\newcommand{\term}{ \mathrm{term} }
\newcommand{\tr}{ \mathrm{tr} }
\newcommand{\mek}{ m^{(k)} }
\newcommand{\smek}{ r^{(k)} }
\newcommand{\h}{ \underline{s}_n}
\newcommand{\Ht}{ \underline{s} }
\newcommand{\Height}{ \text{Ht} }
\newcommand{\Ker}{ \text{Ker} }
\newcommand{\Cr}{\text{cr}}
\newcommand{\ee}{\varepsilon}
\newcommand{\Choice}{ \mathrm{Choice} }
\newcommand{\Skipped}{ \mathrm{Skipped} }

\begin{document}

\title[Operator related to semi-meanders, 
via two-sided $q$-Wick formula]{An operator that relates to 
semi-meander polynomials via a two-sided $q$-Wick formula}

\author{Alexandru Nica}
\thanks{Research supported by a Discovery Grant from NSERC, Canada.}
\address{Alexandru Nica: 
Department of Pure Mathematics, University of Waterloo, 
\newline
Waterloo, Ontario, N2L 3G1, Canada.}
\email{anica@uwaterloo.ca}

\author{Ping Zhong}
\address{Ping Zhong:
Department of Pure Mathematics, University of Waterloo, 
\newline
Waterloo, Ontario, N2L 3G1, Canada, and
\newline
School of Mathematics and Statistics, Wuhan University,
Wuhan, Hubei 430072, China.}
\email{ping.zhong@uwaterloo.ca}

\begin{abstract}
We consider the sequence $( Q_n )_{n=1}^{\infty}$ of semi-meander
polynomials which are used in the enumeration of semi-meandric systems
(a family of diagrams related to the classical stamp-folding problem).
We show that for a fixed $d \in \bN$, $( Q_n (d) )_{n=1}^{\infty}$ 
appears as sequence of moments for a compactly supported probability 
measure $\nu_d$ on $\bR$.  More generally, we consider a sequence of 
two-variable polynomials $( \widetilde{Q}_n )_{n=1}^{\infty}$ related 
to a natural concept of ``self-intersecting semi-meandric system'', 
where the second variable of $\widetilde{Q}_n$ keeps track of the 
crossings of such a system; one has, in particular, that 
$Q_n (t) = \widetilde{Q}_n (t,0)$.  We prove that for fixed $d \in \bN$ 
and $q \in (-1,1)$, $( \widetilde{Q}_n (d,q) )_{n=1}^{\infty}$ can be 
identified as sequence of moments for a compactly supported probability 
measure $\nu_{d;q}$ on $\bR$.  The measure $\nu_{d;q}$ is found as 
scalar spectral measure for an operator $T_{d;q}$ constructed by using 
left and right creation/annihilation operators on the $q$-Fock space 
over $\bC^d$, a deformation of the full Fock space over $\bC^d$ 
introduced by Bo$\dot{\text{z}}$ejko and Speicher.  The relevant 
calculations of moments for $T_{d;q}$ are made by using a two-sided 
version of a (previously studied in the one-sided case) ``$q$-Wick 
formula'', which involves the number of crossings of a pair-partition.
\end{abstract}

\maketitle

\section{Introduction}

A {\em meandric system} of order $n$ is a picture obtained by 
independently drawing two non-crossing pair-partitions (a.k.a. 
``arch-diagrams'') of $\{ 1, \ldots , 2n \}$, one of them above 
and the other one below a horizontal line, as exemplified 
in Figure 1.  The combined arches of the two non-crossing 
pair-partitions create a family of disjoint closed curves which wind 
up and down the horizontal line.  If this family consists of only
one curve going through all of $1, \ldots , 2n$, then 
the meandric system in question is called a {\em meander}.  

$\ $

\begin{center}
  \setlength{\unitlength}{0.3cm}
  \begin{picture}(16,9)
  \thicklines
  \put(1,4){\line(1,0){10}}
  \put(1,4){\oval(2,3)[t]}
  \put(-3,4){\oval(2,3)[t]}
  \put(-1,4){\oval(10,10)[t]}
  \put(1,4){\oval(2,3)[b]}
  \put(1,4){\oval(6,6)[b]}
  \put(1,4){\oval(10,10)[b]}
  \put(1,4){\oval(14,14)[b]}
  \put(7,4){\oval(2,3)[t]}
  \put(1,4){\line(-1,0){10}}
  \end{picture}
\hspace{2cm}
  \begin{picture}(5,9)
  \thicklines
  \put(1,4){\line(1,0){10}}
  \put(3,4){\oval(2,3)[t]}
  \put(3,4){\oval(6,6)[t]}
  \put(3,4){\oval(10,10)[t]}
  \put(1,4){\oval(2,3)[b]}
  \put(1,4){\oval(6,6)[b]}
  \put(1,4){\oval(10,10)[b]}
  \put(1,4){\oval(14,14)[b]}
  \put(-5,4){\oval(2,3)[t]}
  \put(1,4){\line(-1,0){10}}
  \end{picture}


$\ $

$\ $

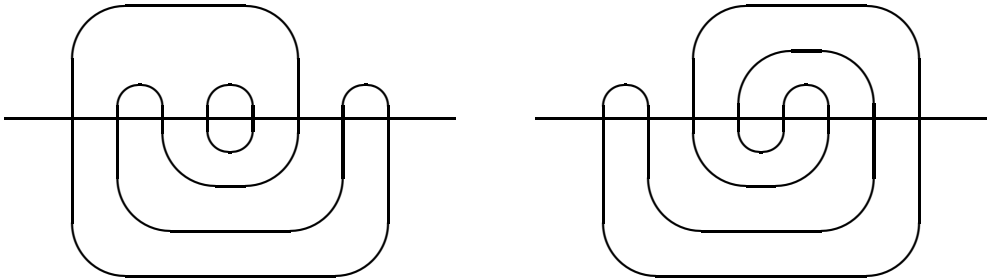
{\bf Figure 1.}  {\em Two meandric systems of order $4$,}

{\em where one of them (on the right) is a meander.}
\end{center}

$\ $

For every $n \in \bN$, a meandric system of order $n$ has
at least $1$ and at most $n$ components; for 
$k \in \{ 1, \ldots , n \}$, we will denote by $\mek_n$
the number of such meandric systems which have $k$ components.
The polynomial 
\begin{equation}    \label{eqn:intro1}
P_n (t)  = \sum_{k=1}^n \mek_n t^k
\end{equation}
is called the $n$th {\em meander polynomial}.  (For instance
$P_1 (t) = t, \ P_2 (t) = 2t + 2t^2$.
Meander polynomials up to degree $12$ can be found in  
\cite[Section 2.3]{DFGG1997}.)  

The problem of enumerating meandric systems, which amounts 
to understanding the coefficients of the above sequence of 
$P_n$'s, turns out to be difficult, and has received a 
substantial amount of interest from the mathematics 
and physics communities; see e.g. Section 4 of the survey 
paper \cite{DF2000}.

An interesting feature of the meander polynomials (\ref{eqn:intro1}) 
is that for certain values of $t \in \bR$, the numerical sequence
$( P_n (t) )_{n=1}^{\infty}$ can be identified as moment sequence 
for a compactly supported probability measure $\mu_t$ on $\bR$.
The typical way of finding the measure $\mu_t$ is as a scalar 
spectral measure for some bounded selfadjoint operator in a 
$C^{*}$-probability space.  The largest range of $t$'s for which 
this can be done appears to be 
$\{ 2 \cos \frac{\pi}{n} \mid n \geq 3 \} \cup [2, \infty )$,
as found in \cite[Section 3]{CJS2014} by using an operator model 
which lives in a planar algebra.  The considerations of the present 
paper bear an analogy with a simpler operator model, which works 
only for integer values of $t$, and was described in \cite{N2016} 
in terms of a free semicircular system of operators, following the 
idea of a random matrix model from \cite{DFGG1997, M1995}.  More 
precisely: for $t=d \in \bN$, the probability measure $\mu_d$ on 
$\bR$ determined uniquely by the moment conditions
\begin{equation}   \label{eqn:intro2}
\int_{\bR} x^n \, d \mu_d (x) = P_n (d)
= \sum_{k=1}^n \mek_n d^k ,
\ \ \forall \, n \in \bN 
\end{equation}
can be described as follows.  We start with
a free family $a_1, \ldots , a_d$ of 
selfadjoint elements in a $C^{*}$-probability space 
$( \cA , \varphi )$, such that every $a_i$ ($1 \leq i \leq d$) 
has centred semicircular distribution of variance $1$.  We then
consider the $C^{*}$-probability space 
$( \cA \otimes \cA , \varphi \otimes \varphi)$, and the 
positive element
\begin{equation}    \label{eqn:intro3}
X_d = (a_1 \otimes a_1  + a_2 \otimes a_2 + \cdots + 
     a_d \otimes a_d)^2 \in \cA \otimes \cA .
\end{equation}
For every $n \in \bN$, the moment 
$( \varphi \otimes \varphi ) (X_d^n)$ turns out to be 
equal to $P_n (d)$ (cf. \cite{N2016}, Proposition 5.9).
Thus the scalar spectral distribution of $X_d$ with respect to 
$\varphi \otimes \varphi$ is precisely the probability measure 
$\mu_d$ from Equation (\ref{eqn:intro2}). 

$\ $

In this paper we study the analogue of (\ref{eqn:intro2})
for {\em semi-meander polynomials}, a sequence of polynomials 
$(\xR_n)_{n=1}^{\infty}$ which are used in connection to the 
enumeration of {\em semi-meandric systems}.  There are several 
equivalent descriptions for what is a semi-meandric system (see
e.g. Section 2.2 of \cite{DFGG1997}).  For our purposes, the most 
convenient way to look at these objects is by identifying them 
with a special class of meandric systems.  More precisely, for 
every $n \in \bN$ let us consider the {\em rainbow} pair-partition
\begin{equation}    \label{eqn:intro4}
\rho_{2n}  := \{ \{1, 2n\}, \{2, 2n-1\}, \ldots, \{n, n+1\} \}, 
\end{equation}
and let $\cR_n$ denote the set of meandric systems of order $n$ 
for which the pair-partition under the horizontal line is constrained
to equal $\rho_{2n}$.  (For example, both the meandric systems 
shown in Figure 1 are from $\cR_4$.)  For every $n \in \bN$ and 
$k \in \{ 1, \ldots , n \}$ we will denote by $\smek_n$ the number of 
meandric systems in $\cR_n$ which have exactly $k$ components.  The
polynomial 
\begin{equation}   \label{eqn:intro5}
\xR_n (t) := \sum_{k=1}^n \smek_n t^k
\end{equation}
is called the $n$th {\em semi-meander polynomial}.  (For instance
$\xR_1 (t)=t, \xR_2(t)=t+t^2$.
A table with numerical data concerning the numbers $\smek_n$ 
can be found in \cite[Section 2.3]{DFGG1997}.)
 
On a historical note, we mention that the meandric systems in $\cR_n$
truly deserve special attention, due to their connection to the old 
(but still open, to our knowledge) problem of enumerating the foldings
of a strip of stamps, which can be tracked back to the treatise on
number theory by Lucas \cite{L1891}.  More precisely: the foldings of a 
strip of $n$ stamps where stamp no.1 stays on top of the folding are 
in bijective correspondence with the connected meandric systems in 
$\cR_n$, and are thus counted by the linear coefficient $r_n^{(1)}$
of the polynomial $Q_n$.  As the topic of the present paper is not 
directly related to foldings, we will not elaborate on that but rather
refer the reader to \cite[Section 2]{DFGG1997} (see also Sections 2 
and 5 of the survey paper \cite{L2013}) for the details of this 
bijective correspondence.

In the present paper we show that, analogously to the discussion 
for meander polynomials: when we set $t = d \in \bN$,
$( Q_n (d) )_{n=1}^{\infty}$ is the moment sequence of a compactly 
supported probability measure $\nu_d$ on $\bR$ which is related to 
free probability.  More precisely, we find $\nu_d$ as the distribution 
of a selfadjoint operator $T_d$ which arises naturally in the framework 
of two-faced free probability theory introduced by Voiculescu in 
\cite{V2014} (a survey of this direction of research, and many references, 
can be found in the expository paper \cite{V2016}).  In order to describe 
$T_d$, we review a bit of terminology concerning creation and annihilation 
operators on the full Fock space $\cT_d$ over $\bC^d$: let 
$\phivac : B( \cT_d ) \to \bC$ denote the vacuum-state on $B( \cT_d ) $, 
and let $L_1, \ldots , L_d$ and $R_1, \ldots , R_d$ denote the 
left and respectively right creation operators on $\cT_d$ associated to 
the vectors $e_1, \ldots , e_d$ from the standard orthonormal basis of 
$\bC^d$.  (These notations are reviewed in more detail in Section 3 below.)

$\ $

{\bf Proposition 1.1.} 
{\em In the framework just described, we put
\begin{equation}   \label{eqn:intro6}
T_d := \sum_{i=1}^d ( L_i + L_i^{*} ) (R_i + R_i^{*} ) \in B( \cT_d ).
\end{equation}
Then $T_d = T_d^{*}$, and has
\begin{equation}   \label{eqn:intro7}
\phivac ( T_d^n ) = Q_n (d), \ \ \forall \, n \in \bN .
\end{equation}  }

$\ $

{\bf Corollary 1.2.} {\em For every $d \in \bN$, 
$( Q_n (d) )_{n=1}^{\infty}$ appears as sequence of moments for 
a compactly supported probability measure $\nu_d$ on $\bR$.
The measure $\nu_d$ can be found as the distribution of the 
operator $T_d$ from Proposition 1.1 with respect to the 
vacuum-state on the full Fock space $\cT_d$. }

$\ $

In the framework of Proposition 1.1, it can be easily verified 
that, for every $1 \leq i \leq d$, the operator 
$( L_i + L_i^{*} ) (R_i + R_i^{*} )$ is selfadjoint and has 
Marchenko-Pastur distribution (the free analogue of the standard 
Poisson distribution) with respect to the vacuum-state.
Thus the operator $T_d$ from (\ref{eqn:intro6}) is the sum of 
$d$ elements with free Poisson distributions -- however, these 
$d$ elements are not ``independent'' in the sense of some 
non-commutative probability theory 
(so it is not clear if moment-cumulant methods could be applied to 
calculate the distribution of $T_d$ by starting from the distribution 
of its $d$ summands). 

We will obtain Proposition 1.1 as the special case of a more general 
result, where we look at a two-variable generalization of the 
semi-meander polynomials $\xR_n$.  Suppose that in the construction 
of a meandric system of order $n$ we actually allow the two 
pair-partitions drawn above and below the horizontal line to run 
in the full set $\cP_2 (2n)$ of all pair-partitions of 
$\{ 1, \ldots , 2n \}$ (that is, we give up the non-crossing 
requirement).  We then get a larger collection of pictures, 
which could be called {\em self-intersecting meandric systems}; 
some examples of such pictures are shown in Figure 2.  A 
self-intersecting meandric system of order $n$ consists of a 
family of $k$ closed curves, with $1 \leq k \leq n$, and where 
now we may also have a number of {\em crossings} in between these curves 
(including the possibility that one of the curves crosses itself, or the 
possibility that two of the curves cross each other multiple times).  
The total number of crossings of a self-intersecting 
meandric system of order $n$ is at least $0$ and at most $n^2 - n$, where
the upper bound is found by noting that every crossing arises either above 
or below the horizontal line, and there can be at most $n$-choose-$2$
crossings of each of these two kinds.  The usual ``meandric systems'' 
discussed above are retrieved, of course, as the self-intersecting 
meandric systems which have $0$ crossings.  (The formal definitions 
of the notions introduced in this paragraph appear in Section 3.4 below 
-- including precise formulas for what what we mean by 
``number of closed curves'' and by ``number of crossings'' for a 
self-intersecting meandric system of order $n$.)

$\ $

\begin{center}
  \setlength{\unitlength}{0.3cm}
  \begin{picture}(16,9)
  \thicklines
  \put(1,4){\line(1,0){10}}   
    \put(-4,4){\oval(4,3)[t]}
        \put(3,4){\oval(2,3)[t]}
            \put(4,4){\oval(8,6)[t]}
                \put(1,4){\oval(10,9)[t]}
                
  \put(-2,4){\oval(4,3)[b]}
  \put(7,4){\oval(2,3)[b]}
  \put(0,4){\oval(4,6)[b]}
  \put(-1,4){\oval(10,9)[b]}
  \put(1,4){\line(-1,0){10}}
  \end{picture}
\hspace{2cm}
  \begin{picture}(5,9)
 \thicklines
  \put(1,4){\line(1,0){10}}
    
    \put(0,4){\oval(4,3)[t]}
        \put(0,4){\oval(8,6)[t]}
            \put(0,4){\oval(12,9)[t]}
                \put(4,4){\oval(8,12)[t]}
  \put(1,4){\oval(2,3)[b]}
  \put(1,4){\oval(6,6)[b]}
  \put(1,4){\oval(10,10)[b]}
  \put(1,4){\oval(14,14)[b]}
  \put(1,4){\line(-1,0){10}}
  \end{picture}


$\ $

$\ $

{\bf Figure 2.}  
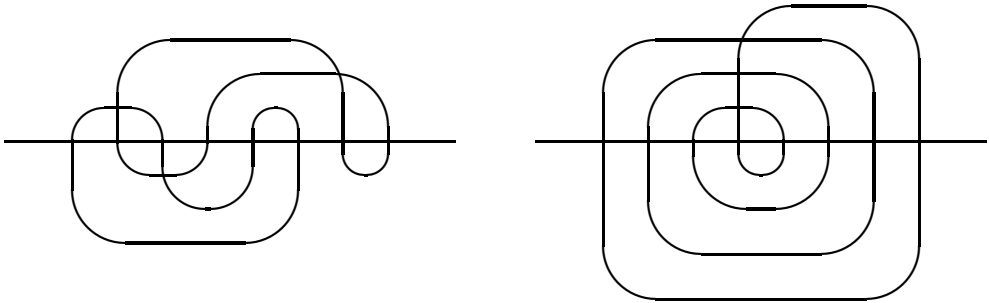
{\em Two self-intersecting meandric systems of 
order $4$, one with  \\
$2$ closed curves and $3$ crossings, and the other 
with $1$ closed curve and $3$ crossings. }
\end{center}

$\ $

{\bf Notation 1.3.}  
(1) For every $n \in \bN$ we let $\widetilde{\cR}_n$ denote the 
set of self-intersecting meandric systems of order $n$ for which the pair-partition
under the horizontal line is constrained to be the rainbow pair-partition 
$\rho_{2n}$ of Equation (\ref{eqn:intro4}).  (For instance the second picture in 
Figure 2 shows a self-intersecting meandric system in 
$\widetilde{\cR}_4 \setminus \cR_4$.) 

(2) For every $n \in \bN$ we consider the polynomial
\begin{equation}   \label{eqn:intro8}
\widetilde{\xR}_n (t,u) := \sum_{k=1}^n  \sum_{\ell = 0}^{n^2 -n} 
r_n^{(k, \ell )} t^k u^{\ell} ,
\end{equation}
where $r_n^{(k, \ell )}$ stands for the number of self-intersecting 
meandric systems in $\widetilde{\cR}_n$ which have exactly $k$ closed 
curves and $\ell$ crossings
\footnote{In the summations from Equation (\ref{eqn:intro8}) 
one may actually restrict the range of $\ell$ to an upper 
bound of $(n^2 - n )/2$, since a system in $\widetilde{\cR}_n$ 
can only have crossings above the horizontal line.}
.  (For instance $\widetilde{\xR}_1 (t,u) = t, \ 
\widetilde{\xR}_2 (t,u) = t(1+u) + t^2$.
Clearly, for every $n \in \bN$, the semi-meander polynomial 
$\xR_n (t)$ can be retrieved as $\widetilde{\xR}_n (t,0)$.)

\vspace{6pt}

Now, the framework of creation/annihilation operators on $\cT_d$ used in 
Proposition 1.1 can be viewed as the special case $q=0$ of a $q$-deformation 
introduced by Bo$\dot{\text{z}}$ejko and Speicher \cite{BS1991} which has 
been well-studied since the 1990's.  For every $q \in (-1, 1)$ one 
has a {\em $q$-Fock space} $\cT_{d;q}$ and one can consider the 
operators $L_{1;q}, \ldots , L_{d;q}$ and $R_{1;q}, \ldots , R_{d;q}$ 
of left and respectively right creation associated to the vectors 
$e_1, \ldots , e_d$ from the standard orthonormal basis of $\bC^d$. 
(The precise definition of these operators is reviewed in Section 
3.3 below.)  The statement of Proposition 1.1 generalizes as follows.

$\ $ 

{\bf Theorem 1.4.} 
{\em Let $d \in \bN$ and $q \in (-1,1)$, and consider the operator  
\begin{equation}   \label{eqn:intro9}
T_{d;q} := \sum_{i=1}^d 
( L_{i;q} + L_{i;q}^{*} ) (R_{i;q} + R_{i;q}^{*} ) 
\in B( \cT_{d;q} ).
\end{equation}
Then $T_{d;q} = T_{d;q}^{*}$, and its moments with respect to the 
vacuum-state $\phivac$ on the $q$-Fock space are
\begin{equation}   \label{eqn:intro10}
\phivac ( T_{d;q}^n ) 
= \widetilde{\xR}_n (d,q), \ \ n \in \bN .
\end{equation}  }

$\ $

{\bf Corollary 1.5.} {\em For every $d \in \bN$ and $q \in (-1,1)$,
$( \widetilde{\xR}_n (d,q) )_{n=1}^{\infty}$ appears as sequence of moments 
for a compactly supported probability measure $\nu_{d;q}$ on $\bR$.  
The measure $\nu_{d;q}$ can be found as the 
distribution of the operator $T_{d;q}$ from Theorem 1.4 with respect 
to the vacuum-state on the $q$-Fock space $\cT_{d;q}$. }

$\ $

{\bf Remark 1.6.}
(1) The main tool for proving the moment formula (\ref{eqn:intro10}) in
Theorem 1.4 is a two-sided version of a $q$-Wick formula (previously 
studied in the one-sided case, cf. \cite{BKS1997, EP2003}) which 
involves the number of crossings of a pair-partition.  The two-sided 
$q$-Wick formula is discussed in Sections 3.2 and 4.1 below.

(2) {\em The case $q=0$.}  Clearly, Proposition 1.1 is the special case
$q=0$ of Theorem 1.4.  The recent body of work on two-faced free 
probability allows for several short proofs of this special case.  
Indeed, Proposition 1.1 is about the $(2d)$-tuple of operators 
$( L_i + L_i^{*} )_{i=1}^d \cup ( R_i + R_i^{*} )_{i=1}^d$ which is the
prototypical example of bi-free Gaussian system appearing in the 
bi-free central limit theorem from \cite{V2014}.  The explicit formulas 
that one has for the joint moments of this $(2d)$-tuple of operators 
can also be read by using the bi-free cumulant theory developped 
in \cite{CNS2015, CNS2015b}, or its precursor focused on canonical 
$(2d)$-tuples from \cite{MN2015}.  Yet another approach to Proposition 1.1 
can be found by combining a matrix model for semi-meander polynomials 
proposed in \cite{DFGG1997} with the bi-free large $N$ limits discussed in 
\cite{S2017}.  We give more details on these alternative proofs in 
Section 4.3 below.

In connection to the case $q=0$ we also mention that, among the 
various patterns that can be prescribed for the bottom part of 
a meandric system, the rainbow pair-partition $\rho_{2n}$ is 
believed to provide the case which is hardest to approach (as 
opposed, for instance, to the case when the bottom part of the 
meandric system is prescribed to be the interval pair-partition 
$\{ \, \{1,2 \}, \ldots , \{ 2n-1 , 2n \} \, \}$ -- see Section 6.3 
of \cite{DFGG1997}, particularly the comment in the last paragraph 
of that section).  In this light, it is quite nice that the semi-meander 
polynomials $Q_n$ can nevertheless be related to moments of operators 
in the two-sided framework.  This seems to be caused by the fortunate 
circumstance that in the rectangular pictures which we use (following 
\cite{CNS2015}) to depict pair-partitions of $\{ 1, \ldots , 2n \}$, 
an important ``labels-to-heights'' permutation of the $2n$ points in 
the picture converts $\rho_{2n}$ into 
$\{ \, \{1,2 \}, \ldots , \{ 2n-1 , 2n \} \, \}$ (see Notation 
\ref{def:33} below, and its follow-up in the proof of 
Lemma \ref{lemma:415}).

(3) On the lines of Theorem 1.4, we observe that the enlarged 
framework of self-intersecting meandric systems also works well 
in connection to the meander polynomials $P_n (t)$ from 
Equation (\ref{eqn:intro1}).  More precisely, 
let us look at the two-variable generalization of these 
polynomials, defined as follows.

$\ $

{\bf Notation 1.7.}
For every $n \in \bN$, consider the polynomial
\begin{equation}   \label{eqn:intro11}
\widetilde{P}_n (t,u) :=  \sum_{k=1}^n  \sum_{\ell = 0}^{n^2 -n} 
m_n^{(k, \ell )} t^k u^{\ell} ,
\end{equation}
where $m_n^{(k, \ell )}$ stands for the number of self-intersecting 
meandric systems of order $n$ which have exactly $k$ closed 
curves and $\ell$ crossings. (For instance 
$\widetilde{P}_1 (t,u) = t, \ 
\widetilde{P}_2 (t,u) = t(2 + 4u) + t^2 (2 + u^2)$.
Clearly, for every $n \in \bN$, the meander polynomial $P_n (t)$ can be 
retrieved as $\widetilde{P}_n (t,0)$.) 

\vspace{6pt}

It is natural to ask if some $q$-deformation of the operator 
$X_d$ from Equation (\ref{eqn:intro3}) could allow us to infer 
that $( \widetilde{P}_n (d,q) )_{n=1}^{\infty}$ 
appears as sequence of moments for a probability measure 
$\mu_{d;q}$ on $\bR$, for $d \in \bN$ and $q \in (-1,1)$.  There 
is in fact an obvious candidate for how to do the $q$-deformation 
of $X_d$: the semicircular elements $a_i$ appearing in 
(\ref{eqn:intro3}) can be concretely realized as 
$L_i + L_i^{*} \in B( \cT_d )$, and can then be $q$-deformed 
to $L_{i;q} + L_{i;q}^{*} \in B( \cT_{d;q} )$, $1 \leq i \leq d$.  
It is easy to see (by a straightforward adjustment of the argument 
shown in \cite[Proposition 5.9]{N2016} for the case $q=0$) that 
this candidate of $q$-deformation does indeed the required job.  
That is, we have the following proposition.

$\ $

{\bf Proposition 1.8.}
{\em Let $d \in \bN$ and $q \in (-1,1)$, and consider the 
selfadjoint operator  
\begin{equation}   \label{eqn:intro12}
X_{d;q} := \Bigl( \, \sum_{i=1}^d 
( L_{i;q} + L_{i;q}^{*} ) \otimes (L_{i;q} + L_{i;q}^{*} ) \, \Bigr)^2
\in B( \cT_{d;q} ) \otimes B( \cT_{d;q} ) .
\end{equation}
The moments of $X_{d;q}$ with respect to the state 
$\phivac \otimes \phivac$ on $B( \cT_{d;q} ) \otimes B( \cT_{d;q} )$ are
\begin{equation}   \label{eqn:intro13}
( \phivac \otimes \phivac ) ( X_{d;q}^n ) 
= \widetilde{P}_n (d,q), \ \ n \in \bN .
\end{equation}  }

$\ $

{\bf Corollary 1.9.} {\em For every $d \in \bN$ and $q \in (-1,1)$,
$( \widetilde{P}_n (d,q) )_{n=1}^{\infty}$ appears as sequence 
of moments for a compactly supported probability measure 
$\mu_{d;q}$ on $\bR$.  The measure $\mu_{d;q}$ can be found as 
the distribution of the operator $X_{d;q}$ from Proposition 1.8 
with respect to the state 
$\phivac \otimes \phivac$ on 
$B ( \cT_{d;q} ) \otimes B( \cT_{d;q} )$. }

$\ $

{\bf Organization of the paper.}
Besides the present introduction, the paper has three sections.
Section 2 covers the combinatorics relevant for proving the 
two-sided $q$-Wick formula.  In Section 3 we prove Theorem 1.4; a good
part of the section is devoted to establishing the special case 
of the two-sided $q$-Wick formula which is used in the proof of the 
theorem.  The final Section 4 discusses some miscellaneous remarks 
related to the results of the paper: in Section 4.1 we 
complete the discussion of the two-sided $q$-Wick formula, in Section 4.2
we prove Proposition 1.8, and in Section 4.3 we make some comments 
related to the special case $q=0$.

\section{Pair-partitions and strings of symbols from 
$\boldmath{ \{ 1,* \} }$ }
 
\setcounter{equation}{0}

\subsection{Pair-partitions}

\begin{definition}   \label{def:31}
Let $n$ be a positive integer.

\vspace{6pt}

(1) We denote by $\mathcal{P}_2(2n)$ the set of all 
{\em pair-partitions} of $\{ 1, \ldots , 2n \}$. 
A $\pi \in \mathcal{P}_2(2n)$ is thus of the
form $\pi=\{ V_1, \ldots, V_n \}$, where the sets 
$V_1, \ldots, V_n$ (called {\em pairs}, or {\em blocks} 
of $\pi$) satisfy:
$\cup_{i=1}^n V_i = \{ 1, \ldots , 2n \}, 
\ \ V_i \cap V_j = \emptyset \mbox{ for } i \neq j,
\mbox{ and } |V_1| =  |V_2| = \cdots = |V_n| = 2$.

\vspace{6pt}

(2) Let $\pi = \{ V_1, \ldots , V_n \}$ be in 
$\mathcal{P}_2(2n)$.  We will say that  
two distinct blocks $V_i, V_j$ of $\pi$
are {\em crossing} to mean that upon writing 
$V_i \cup V_j = \{a, b, c, d \} 
\quad \text{with}\quad a<b<c<d$,
one finds $a,c$ in one of the two blocks and $b,d$ 
in the other.  The {\em number of crossings} of 
$\pi$ is defined as
\[ 
\Cr (\pi) := \ \vline \ \{ (i,j) \mid
1 \leq i < j \leq n, \ V_i \mbox{ crosses } V_j \} \ \vline \ .
\]
A pair-partition $\pi \in \cP_2 (2n)$ is said to 
be {\em non-crossing} when it has $\Cr ( \pi ) = 0$. 
The collection of all non-crossing pair-partitions in 
$\cP_2 (2n)$ will be denoted by $\NC_2(2n)$.
\end{definition}

\begin{remark}   \label{rem:32}
In the rather extensive literature pertaining to 
pair-partitions, one finds their pictures drawn either in 
``linear'' representation (with $2n$ points labelled by
$1, \ldots , 2n$ depicted along a line) or in ``circular'' 
representation (with $2n$ points depicted around a circle). 
In either representation, $\Cr ( \pi )$ appears as the 
number of intersections between the curves drawn in 
the picture in order to represent the pairs of $\pi$.

\begin{center}
\[\begin{tikzpicture}[baseline]
\foreach \y in{0,...,9}{
\pgfmathtruncatemacro{\nodename}{\y+1}
\node (ball\nodename) [draw, circle, inner sep=0.07cm] at (\y-2, 0) {};
\node at (\y-2, 0.5){\nodename};
}
\draw[rline, thick](ball4)++(0,-0.35)-- ++(0,-0.5)--++(1,0)--++(0,0.5);
\draw[rline, thick](ball6)++(0,-0.35)-- ++(0,-0.5)--++(2,0)--++(0,0.5);
\draw[rline, thick](ball2)++(0,-0.35)-- ++(0,-1)--++(5,0)--++(0,1);
\draw[rline, thick](ball1)++(0,-0.35)-- ++(0,-1.5)--++(8,0)--++(0,1.5);
\draw[rline, thick](ball3)++(0,-0.35)-- ++(0,-2)--++(7,0)--++(0,2);

\end{tikzpicture}
\]

{\bf Figure 3.}  
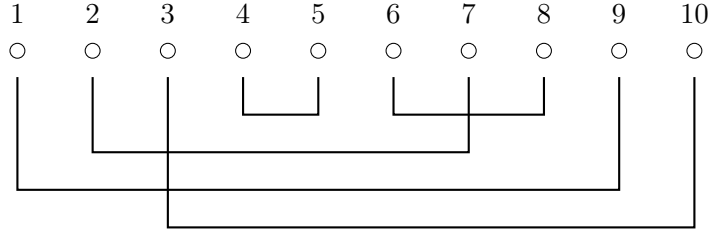
{\em Linear representation of the pair-partition}
 
$\pi=\{\{1,9\},\{2,7\},\{3,10 \},\{4,5 \},\{6,8 \} \}
\in \cP_2 (10)$.
\end{center}

In this paper we will use pictures drawn in circular 
representation but where, following \cite{CNS2015}, our ``circle'' 
will in fact be a rectangle, and the $2n$ points labelled by
$1, \ldots , 2n$ will be depicted on the vertical sides of the 
rectangle: labels $1, \ldots , n$ on the {\em left} side of 
the rectangle (running downwards) and labels $n+1, \ldots , 2n$ 
on the {\em right} side of the rectangle (running upwards).

For a concrete example, in Figures 3 and 4 we consider the 
pair-partition 
\[
\pi=\{\{1,9\},\{2,7\},\{3,10 \},\{4,5 \},\{6,8 \} \}
\in \cP_2 (10),
\]
which has $\Cr ( \pi ) = 3$. Figure 3 shows the linear 
representation of this $\pi$, and Figure 4 shows its 
``rectangular'' representation.

\begin{center}
\[
\begin{tikzpicture}[baseline]
		\draw[thick, rline] (-2,0.25) -- (-2, -4.5);
		\draw[thick, dashed](-2,-4.5) -- (2,-4.5) ;
		\draw[thick,rline](2,-4.5)-- (2,0.25);
		\draw[thick,dashed](2,0.25)--(-2,0.25);

		\foreach\x in {1,..., 5}
		{
					\pgfmathtruncatemacro{\nodename}{\x}
		  \node[left]at (-2.2,0.4-\x*.8) {\x};
		  \node(ball\nodename)[draw, circle, inner sep=0.07cm] at (-2,0.4-\x*.8) {}; 
		};
  		\foreach\x in {6,..., 10}
		{
							\pgfmathtruncatemacro{\nodename}{\x}
		  \node[left]at (2.7,+\x*.8-8.8) {\x};
		  		  \node(ball\nodename)[draw, circle, inner sep=0.07cm] at (2,+\x*.8-8.8) {}; 
		};
                \draw[rline,thick](ball1)--++(2,0)--++(0,-1.2)--(ball9);
                                \draw[rline,thick](ball2)--++(1.3,0)--++(0,-2)--(ball7);
               \draw[rline,thick](ball4)--++(0.7,0)--++(0,-0.8)--(ball5);
               \draw[rline,thick](ball3)--++(2.7,0)--++(0,1.2)--(ball10);
               \draw[rline,thick](ball8)--++(-0.7,0)--++(0,-1.6)--(ball6);

\end{tikzpicture}
\]
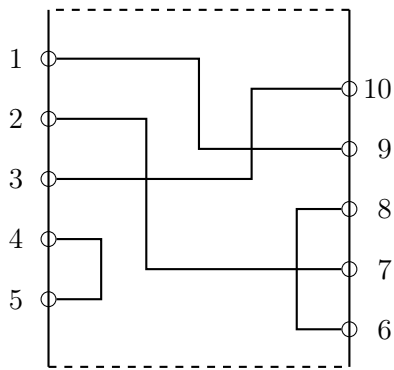
{\bf Figure 4.}  {\em Rectangular representation of the same
$\pi \in \cP_2 (10)$ as in Figure 3.}
\end{center}

 
In our rectangular pictures it will be important (again following \cite{CNS2015}) 
to keep track of the relative heights of the $2n$ labelled points.  Nearly 
everywhere in this paper (with the exception of Section 4.1) we only need to 
look at the situation where these heights {\em alternate} between the left and 
the right side of the rectangle.  The convention for measuring heights which 
arises from \cite{CNS2015} is that we set the height level $0$ at the top 
horizontal side of the rectangle, and we measure heights {\em downwards} from 
there.  Thus in Figure 4 we get, with ``$\Height$'' for 
height 
\footnote{ Even though they are measured downwards, we will nevertheless refer
to the distances measured from the top horizontal side of the rectangle by 
calling them ``heights''. }
and by 
writing the height as a function of the corresponding label:
\begin{center}
$\Height(1)=1, \Height(2)=3, \ldots , \Height(5)=9$ and
$\Height(6)=10, \Height(7)=8, \ldots , \Height(10)=2$.
\end{center}
In general, for the ``height in terms of label'' map we will use the following 
notation.  (We mention that ``$\Ht_n$'' is a special case of a rather established 
notation for a permutation ``$\Ht_{\chi}$'' which will be reviewed in Section 4.1 
below, and plays an essential role in the combinatorics of two-faced free probability.)
\end{remark}

\begin{notation}   \label{def:33}
For every $n \in \bN$, we will denote by $\Ht_n$ the permutation of
$\{ 1, \ldots , 2n \}$ defined (in the usual two-line notation for 
permutations) as
\begin{equation}   \label{eqn:33a}
\Ht_n := \left(  \begin{array}{cccccccc}
1   & 2  & \cdots & n    & n+1 & \cdots & 2n-1  & 2n     \\
1   & 3  & \cdots & 2n-1 & 2n  & \cdots & 4     & 2 
\end{array}  \right) .
\end{equation}
Thus $\Ht_n (k)$ is the height of the point with label $k$ ($1 \leq k \leq 2n$) 
in the rectangular representation of any $\pi \in \cP_2 (2n)$.

We will also consider the natural action of $\Ht_n$ on $\cP_2 (2n)$, defined by
\[
\left( \pi = \{ V_1, \ldots , V_n \} \right) \ \Rightarrow
\ \left( \Ht_n \cdot \pi = \{ \Ht_n ( V_1), \ldots , \Ht_n (V_n) \} \right) .
\]
It is relevant to observe that the rainbow pair-partition
$\rho_{2n}$ from Equation (\ref{eqn:intro4}) of the Introduction 
is transformed by the action of $\Ht_n$ into an interval pair-partition,
\begin{equation}   \label{eqn:33b}
\Ht_n \cdot \rho_{2n} = \{ \, \{ 1,2 \}, \, \{ 3,4 \}, \ldots ,
\, \{ 2n-1, 2n \} \, \} \in \cP_2 (2n).
\end{equation}
\end{notation}  

$\ $

\subsection{Tuples in \boldmath{$\{ 1,* \}^{2n}$} and the map
\boldmath{$\Phi_n : \cP_2 (2n) \to \cD_{(1,*)} (2n)$}}

$\ $

\noindent
We now start to look at strings made with the symbols ``$1$'' and 
``$*$''.  It will be convenient to view such a string (i.e. tuple 
$\ee \in \{ 1,* \}^m$ for some $m \in \bN$) as a map
$\ee : \{1, \ldots , m \} \to \{ 1,* \}$.

\begin{definition}   \label{def:34}
A $(2n)$-tuple $\ee :\{1, \ldots, 2n \}\rightarrow \{1, *\}$ is 
said to have the {\em Dyck property} when it satisfies the 
inequalities
\begin{equation}   \label{eqn:34a}
\left|\{ 1\leq i\leq h \mid  \ee (i)=1 \}\right|
\geq |\{  1\leq i \leq h \mid  \ee (i)=* \}|,\quad \forall \, 1\leq h \leq 2n,
\end{equation}
with equality when $h=2n$, that is
\begin{equation}   \label{eqn:34b}
|\{ 1\leq i\leq 2n \mid \ee (i)=1 \}|
\ = \ n \ = \ |\{  1\leq i \leq 2n  \mid \ee (i)=* \}|.
\end{equation}
The collection of all the $(2n)$-tuples with the Dyck property 
will be denoted by $\cD_{(1,*)}(2n)$.
 \end{definition}

\begin{remark}   \label{rem:35}
The use of the term ``Dyck property'' in Definition \ref{def:34} is 
justified by a connection to lattice paths.  To every 
$\ee \in \{ 1,* \}^{2n}$ one can associate a path with $2n$ steps in 
$\bZ^2$ which starts at $(0,0)$ and proceeds according to the following
rule:

\noindent
-- for every $1 \leq h \leq 2n$ such that $\ee (h) = 1$ we perform a 
step of $(1,1)$ (North-East step);

\noindent
-- for every $1 \leq h \leq 2n$ such that $\ee (h) = *$ we perform a 
step of $(1,-1)$ (South-East step).

\noindent
That is, our path visits succesively the lattice points 
$(0,0), (1,p_1), (2,p_2), \ldots , (2n, p_{2n})$, where 
for $1 \leq h \leq 2n$ we put
$p_h = \ \vline \ \{ 1 \leq i \leq h \mid \ee (i) = 1 \} \ \vline \
- \ \vline \ \{ 1 \leq i \leq h \mid \ee (i) = * \} \ \vline \ .$
Clearly, the conditions (\ref{eqn:34a}), (\ref{eqn:34b}) in Definition 
\ref{def:34} have the meaning that our lattice path never goes under 
the horizontal axis of $\bZ^2$, and ends at the point $(2n,0)$ on that axis.  
These conditions constitute precisely the definition of a {\em Dyck path} 
in $\bZ^2$.  
\end{remark}

$\ $

\begin{prop-and-notation}   \label{def:36}
Let $n$ be a positive integer.

\vspace{6pt}

(1) Let $\pi$ be in $\cP_2 (2n)$.  We consider the pair-partition 
$\Ht_n \cdot \pi$ (as in Notation \ref{def:33}), and we write it 
explicitly, $\Ht_n \cdot \pi = \{ W_1, \ldots , W_n \}$.
Let $\ee : \{ 1, \ldots , 2n \} \to \{ 1,* \}$ be defined by 
\begin{equation}   \label{eqn:36a}
\left\{   \begin{array}{l}
\ee ( \min (W_1) ) = \cdots = \ee (\min (W_n) ) = 1, \mbox{ and }  \\
\ee ( \max (W_1) ) = \cdots = \ee (\max (W_n) ) = *. 
\end{array}  \right.
\end{equation} 
Then $\ee \in \cD_{(1,*)} (2n)$.

\vspace{6pt}

(2) We will denote by $\Phi_n : \cP_2 (2n) \to \cD_{(1,*)} (2n)$
the map which associates to every $\pi \in \cP_2 (2n)$ the tuple
$\ee \in \cD_{(1,*)} (2n)$ described in part (1) above.
\end{prop-and-notation}

\vspace{6pt}

{\em Proof}  
that $\ee$ of (\ref{eqn:36a}) has the Dyck property.
Fix an $h \in \{ 1, \ldots , 2n \}$ and observe that
\[
\vline \, \{ 1 \leq i \leq h \mid \ee (i) = 1 \} \, \vline 
\ = \ \vline \ \{ 1 \leq i \leq h \mid 
                  i \in \{ \min (W_1), \ldots , \min (W_n) \} \ \vline 
\]
\[
= \ \vline \ \{ \min (W_1), \ldots , \min (W_n) \} 
                  \cap \{ 1, \ldots , h \} \ \vline                    
\ = \ \vline \ \{ 1 \leq m \leq n \mid \min (W_m) \leq h \} \ \vline \ .
\]
Similarly, we see that
$\vline \ \{ 1 \leq i \leq h \mid \ee (i) = * \} \ \vline \
= \ \vline \, \{ 1 \leq m \leq n \mid \max (W_m) \leq h \} \ \vline$ .
Thus (\ref{eqn:34a}) amounts to
$\vline \ \{ 1 \leq m \leq n \mid \min (W_m) \leq h \} \ \vline \ 
\geq \ \vline \, \{ 1 \leq m \leq n \mid \max (W_m) \leq h \} \ \vline \ ,$
and holds true due to the obvious implication 
``$( \max (W_m) \leq h ) \Rightarrow ( \min (W_m) \leq h )$''.  The equality
(\ref{eqn:34b}) is also clear, from how $\ee$ is defined in (\ref{eqn:36a}).
\hfill  $\blacksquare$

\begin{example}  \label{example:37}
Consider again the pair-partition $\pi \in \cP_2 (10)$ depicted in Figure 4, 
and let us determine what is $\Phi_5 ( \ee ) \in \cD_{(1,*)} (10)$.  We thus 
look at the {\em heights} of the $10$ points marked around the rectangle,
and: whenever two points at heights $i$ and $j$ are paired, with $i < j$, 
we assign $\ee (i) = 1$ and $\ee (j) = *$.  Formally, we write:
\[
\begin{array}{lll}
\pi & =  & \bigl\{ \, \{ 1,9 \}, \{ 2,7 \}, \{ 3,10 \}, 
                                 \{ 4,5 \}, \{ 6, 8 \}  \, \bigr\}     \\
    & \Rightarrow  & \Ht_{10} ( \pi ) = \bigl\{ \, \{ 1,4 \}, 
           \{ 2,5 \}, \{ 3,8 \}, \{ 6,10 \}, \{ 7, 9 \} \, \bigr\}     \\
    & \Rightarrow  & \ee := \Phi_5 ( \pi ) \mbox{ has }
                     \ee^{-1} (1) = \{ 1,2,3,6,7 \} \mbox{ and }
                     \ee^{-1} (*) = \{ 4,5,8,9,10 \}                    \\
    & \Rightarrow  & \Phi_5 ( \pi ) \mbox{ is the tuple }
                     \ee = (1,1,1,*,*,1,1,*,*,*).
\end{array}
\]
\end{example}

\begin{remark}   \label{rem:37}
For the subsequent discussion and figures, it will come in handy to 
give names to the actual points (geometric entities) marked on the boundary 
of the rectangle: we will denote them as $P_1, \ldots , P_{2n}$, in such a 
way that
\[
\mbox{ (label of $P_{k}$) } = k \mbox{ and }
\mbox{ (height of $P_{k}$) } = \Ht_n ( k ), \ \ 
\mbox{ for $1 \leq k \leq 2n$.}
\]

\begin{center}
\[
\begin{tikzpicture}[baseline]
		\draw[thick, rline] (-2,0.25) -- (-2, -4.5);
		\draw[thick, dashed](-2,-4.5) -- (2,-4.5) ;
		\draw[thick,rline](2,-4.5)-- (2,0.25);
		\draw[thick,dashed](2,0.25)--(-2,0.25);

		\foreach\x in {1,..., 5}
		{
					\pgfmathtruncatemacro{\nodename}{\x}
		  \node[left]at (-2.2,0.4-\x*.8) {$P{_\x}$};
		  \node(ball\nodename)[draw, circle, inner sep=0.07cm] at (-2,0.4-\x*.8) {}; 
		};
		\node [left] at (-2.2-1, -0.4){deco.$=1$};
		\node [left] at (-2.2-1, -0.4-0.8){deco.$=1$};
		\node [left] at (-2.2-1, -0.4-1.6){deco.$=*$};
		\node [left] at (-2.2-1, -0.4-2.4){deco.$=1$};
		\node [left] at (-2.2-1, -0.4-3.2){deco.$=*$};
		
  		\foreach\x in {6,..., 10}
		{
							\pgfmathtruncatemacro{\nodename}{\x}
		  \node[left]at (2.7+0.2,+\x*.8-8.8) {$P_\x$};
		  		  \node(ball\nodename)[draw, circle, inner sep=0.07cm] at (2,+\x*.8-8.8) {}; 
		};
		 \node[left]at (2.7+0.2+2,+6*.8-8.8){deco.$=*$};
		  \node[left]at (2.7+0.2+2,+7*.8-8.8){deco.$=*$};
		  \node[left]at (2.7+0.2+2,+8*.8-8.8){deco.$=1$};
		  \node[left]at (2.7+0.2+2,+9*.8-8.8){deco.$=*$};
  		  \node[left]at (2.7+0.2+2,+10*.8-8.8){deco.$=1$};

\end{tikzpicture}
\]
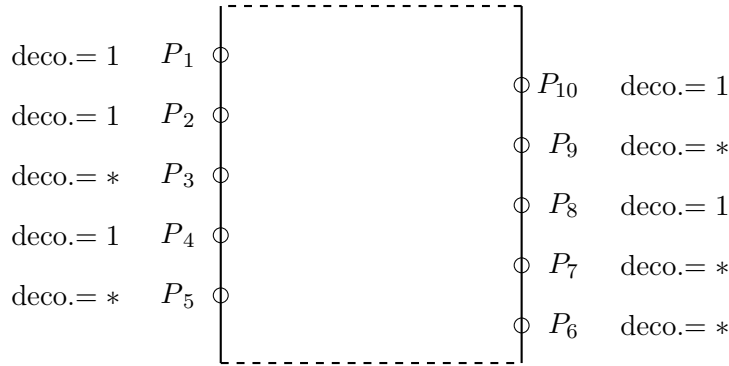
{\bf Figure 5.}  {\em $10$ decorated points $P_1, \ldots , P_{10}$ in a 
rectangular picture.}
\end{center}

In terms of these points $P_1, \ldots , P_{2n}$, the assignments of the 
form ``$\ee (i) = 1$'' or ``$\ee (j) = *$'' introduced in Notation 
\ref{def:36} will be referred to by saying that 
{\em the point at height $i$ is ``$1$''-decorated} and respectively
that {\em the point at height $j$ is ``$*$''-decorated}.  
For illustration, Figure 5 shows the points $P_1, \ldots , P_{10}$
and their decorations, as they come out of Example \ref{example:37}
(e.g. $P_2$ has label $2$, height $3$, and is ``$1$''-decorated, while 
$P_6$ has label $6$, height $10$, and is ``$*$''-decorated).
\end{remark}
 
$\ $

\subsection{Choice numbers and the enumeration of 
\boldmath{$\Phi_n^{-1} ( \ee )$}}

\begin{remark-and-notation}    \label{rem:38}
Let $n$ be a positive integer and let us fix a tuple 
$\ee \in \cD_{(1,*)} (2n)$.  For the proof of the two-sided $q$-Wick 
formula in the next section, it will be important to
have a good description of the pre-image 
$\Phi_n^{-1} ( \ee ) \subseteq \cP_2 ( 2n)$, where $\Phi_n$ is 
the map introduced in Notation \ref{def:36}.  To this end, 
we will make the following (ad-hoc) definition: 
for 
\footnote{ Note that $1 \not\in \ee^{-1} (*)$, due to the 
assumption that $\ee$ has the Dyck property.}
every 
$h \in \ee^{-1} (*) \subseteq \{ 2, \ldots , 2n \}$, we call
{\em ``choice number of $\ee$ at $h$''} the number 
\begin{equation}  \label{eqn:38a}
\Choice_{\ee} (h)
:=  \ \vline \, \{ 1 \leq i \leq h-1 \mid \ee (i) = 1 \} \ \vline \
-  \ \vline \ \{ 1 \leq i \leq h-1 \mid \ee (i) = * \} \ \vline \ . 
\end{equation}
By invoking the Dyck property satisfied by $\ee$, we infer that 
$\Choice_{\ee} (h) \geq 1$.  Indeed:
\begin{align*}
0 
& \leq \ \vline \, \{ 1 \leq i \leq h \mid \ee (i) = 1 \} 
  \ \vline \ - \ \vline \ \{ 1 \leq i \leq h \mid \ee (i) = * \} \ \vline   \\
& = \ \vline \, \{ 1 \leq i \leq h-1 \mid \ee (i) = 1 \} \ \vline \
- \Bigl( \ \vline \ \{ 1 \leq i \leq h-1 \mid \ee (i) = * \} \ \vline  
         \ + 1 \ \Bigr)                                                   \\
& = \Choice_{\ee} (h) -1.
\end{align*}

The rationale for the term ``choice'' used in (\ref{eqn:38a}) 
is that the partitions in $\Phi_n^{-1} ( \ee )$ are naturally
parametrized by ``tuples of choices'' of the form
\begin{equation}  \label{eqn:38c}
( \xlambda_h )_{h \in \ee^{-1} (*)}, \mbox{ with }
\xlambda_h \in [ 1, \Choice_{\ee} (h) ] \cap \bN , 
\ \ \forall \, h \in \ee^{-1} (*).
\end{equation}
Note that, as an obvious consequence of (\ref{eqn:38c}), one has
\begin{equation}  \label{eqn:38b}
| \, \Phi_n^{-1} ( \ee ) \, |  =
\prod_{h \in \ee^{-1} (*)} \Choice_{\ee} (h).
\end{equation}

The procedure of retrieving a pair-partition $\pi \in \cP_2 (2n)$ 
from the information provided by $\ee$ and a tuple of $\xlambda_h$'s 
as in (\ref{eqn:38c}) is rather standard in the literature on 
lattice paths, with the slight difference in terminology that instead 
of $\ee$ one usually refers to the corresponding Dyck path 
mentioned in Remark \ref{rem:35}.  More precisely: a tuple of
$\xlambda_h$'s as indicated in (\ref{eqn:38c}) can be viewed as an
additional piece of structure imposed on the Dyck path corresponding to 
$\ee$, and this yields the notion of ``weighted Dyck path'', which 
is thoroughly studied e.g. in Section 5.2 of the monograph \cite{GJ1983}.

In order to explain how a tuple as in (\ref{eqn:38c}) parametrizes a
pair-partition from $\Phi_n^{-1} ( \ee )$, we find it more illuminating 
to discuss a relevant concrete example, where the presentation can 
be illustrated with pictures.
\end{remark-and-notation}

$\ $

\begin{example}   \label{example:39}
Suppose that $n = 5$ and that 
$\ee = ( 1,1,1,*,*,1,1,*,*,* ) \in \cD_{(1,*)} (10)$.
Thus $\ee^{-1} (*)$ is $\{ 4,5,8,9,10 \}$, and the corresponding 
choice-numbers are:
\[
\Choice_{\ee} (4) = 3-0 = 3, \ \Choice_{\ee} (5) = 3-1 = 2,
\, \Choice_{\ee} (8) = 5-2 = 3, 
\]
\[
\Choice_{\ee} (9) = 5-3 = 2, \ \Choice_{\ee} (10) = 5-4 = 1.
\]
For this $\ee$, we consider the set of pair-partitions 
$\Phi_5^{-1} ( \ee ) \subseteq \cP_2 (10)$, and we will discuss its
enumeration.

Let us draw points $P_1, \ldots , P_{10}$ around a rectangle, in the same 
way as in Figure 5, and for every $h \in \{ 1, \ldots , 10 \}$ let us 
decorate the point at height $h$ by the symbol $\ee (h)$ (this
also is exactly as in Figure 5).  The enumeration of 
$\Phi_5^{-1} ( \ee )$ then amounts to the enumeration of all the 
pair-partitions of $P_1, \ldots , P_{10}$ which have the 
following
property: 
every chord of the pair-partition must connect a point decorated 
as ``$1$'' to a point decorated as ``$*$'', where the 
``$*$''-decorated endpoint has a bigger height than the ``$1$''-decorated 
endpoint of the chord.

In order to draw a generic pair-partition of $P_1, \ldots , P_{10}$ which 
has the required property, we proceed in five steps, as follows: we visit 
the five ``$*$''-decorated points among $P_1, \ldots , P_{10}$, in increasing 
order of their heights, and for every such point $P_m$ we choose its pair 
out of the ``$1$''-decorated points with heights smaller than the one of $P_m$,
and which have not already been used in a preceding step.

To be specific: the first step of our construction is to visit the 
point at height $4$, and to choose a pair for it; we have three 
possibilities for doing so (which corresponds to the fact that
$\Choice_{\ee} (4) = 3$), namely, we can choose any of the points at heights
$1$, $2$ or $3$.  Let's say that we choose the pair to be at height $3$, 
as shown in Figure 6.  The second step of our construction is then to
go to the point at height $5$, and choose a pair for that one; we have 
two possibilities for doing so, corresponding to the fact that 
$\Choice_{\ee} (5) = 3-1 =2$ (indeed, there are three points with 
decoration ``$1$'' and with heights $< 5$, but one of them is
already engaged in a different pair of the construction).  So we must
pair the point at height $5$ with one of the points at heights $1$ or $2$ 
-- let's say we choose the one at height $1$.

\begin{center}
\[
\begin{tikzpicture}[baseline]
		\draw[thick, rline] (-2,0.25) -- (-2, -4.5);
		\draw[thick, dashed](-2,-4.5) -- (2,-4.5) ;
		\draw[thick,rline](2,-4.5)-- (2,0.25);
		\draw[thick,dashed](2,0.25)--(-2,0.25);

		\foreach\x in {1,..., 5}
		{
					\pgfmathtruncatemacro{\nodename}{\x}
		  \node[left]at (-2.2,0.4-\x*.8) {$P{_\x}$};
		  \node(ball\nodename)[draw, circle, inner sep=0.07cm] at (-2,0.4-\x*.8) {}; 
		};
		\node [left] at (-2.2-1, -0.4){deco.$=1$};
		\node [left] at (-2.2-1, -0.4-0.8){deco.$=1$};
		\node [left] at (-2.2-1, -0.4-1.6){deco.$=*$};
		\node [left] at (-2.2-1, -0.4-2.4){deco.$=1$};
		\node [left] at (-2.2-1, -0.4-3.2){deco.$=*$};
		
  		\foreach\x in {6,..., 10}
		{
							\pgfmathtruncatemacro{\nodename}{\x}
		  \node[left]at (2.7+0.2,+\x*.8-8.8) {$P_\x$};
		  		  \node(ball\nodename)[draw, circle, inner sep=0.07cm] at (2,+\x*.8-8.8) {}; 
		};
		\node[left]at (2.7+0.2+2,+6*.8-8.8){deco.$=*$};
		  \node[left]at (2.7+0.2+2,+7*.8-8.8){deco.$=*$};
		  \node[left]at (2.7+0.2+2,+8*.8-8.8){deco.$=1$};
		  \node[left]at (2.7+0.2+2,+9*.8-8.8){deco.$=*$};
		  \node[left]at (2.7+0.2+2,+10*.8-8.8){deco.$=1$};

              \draw[rline,thick](ball1)--++(0.7,0)--++(0,-1.6)--(ball3);
              \draw[rline,thick](ball2)--++(2,0)--++(0,-0.4)--(ball9);

\end{tikzpicture}
\]

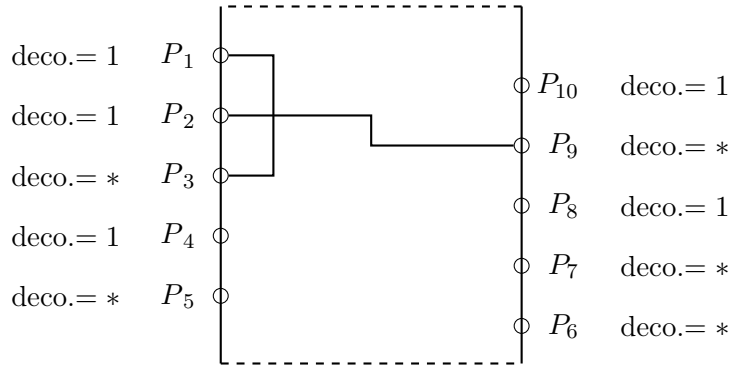
{\bf Figure 6.}  {\em First two steps in the construction of a 
pair-partition 

in $\Phi_5^{-1} ( \ee )$, for the $\ee$ of Figure 5.}
\end{center}


\vspace{6pt}

\noindent
Continuing in the same way, we will then visit the points at heights 
$8$, $9$ and $10$, and choose pairs for them (in three possible ways for the 
point $P_7$ at height $8$, then in two possible ways for the point $P_5$ at 
height $9$, and in a unique way for $P_6$ at height $10$).  We hope this discussion 
convinces the reader that, in our running example, the set $\Phi_5^{-1} ( \ee )$ 
has $3 \cdot 2 \cdot 3 \cdot 2 \cdot 1 = 36$ pair-partitions, as claimed by the
formula (\ref{eqn:38b}).

Let us also describe precisely what are the parameters 
$( \xlambda_h )_{h \in \ee^{-1} (*)}$ mentioned in (\ref{eqn:38c}) which govern
our succesive choices of pairs for the ``$*$''-decorated points in the picture.
We will use the convention that: whenever we want to assign a pair to a 
``$*$''-decorated point on the left side of the rectangle, the possible 
choices of pairs are considered in clockwise order; while for ``$*$''-decorated 
points which lie on the right side of the rectangle, the possible choices of pairs 
are considered in counterclockwise order (so we always follow the boundary of the 
rectangle in the direction going towards the top horizontal side).  

For instance, the first step of the construction presented earlier in this example
assigned a pair to the point at height $4$; the choice of the pair 
was governed by a parameter $\xlambda_4 \in \{ 1,2,3 \}$, and our convention 
says that (for what was picked in Figure 6) we had 
$\xlambda_4=3$. 
Then the second step of the construction was to choose a pair
for the point at height $5$, governed by a parameter $\xlambda_5 \in \{ 1,2 \}$,
and our convention says that (for what was picked in Figure 6) we had 
$\xlambda_5 = 1$.  Figure 7 shows how the pair-partition
$\pi \in \Phi_5^{-1} ( \ee )$ looks like when the choices 
$\xlambda_4= 3$, $\xlambda_5 = 1$ from Figure 6 are followed by 
$\xlambda_8 = \xlambda_9 = 2$ and $\xlambda_{10} = 1$.  
(Or referring to Figure 4, which has the same $\ee$ as in this example: the 
sequence of choices for the pair-partition in that figure is 
$\xlambda_4 = \xlambda_5 =  \xlambda_8 = 2$, $\xlambda_9 = \xlambda_{10} = 1$.) 


\begin{center}
\[
\begin{tikzpicture}[baseline]
		\draw[thick, rline] (-2,0.25) -- (-2, -4.5);
		\draw[thick, dashed](-2,-4.5) -- (2,-4.5) ;
		\draw[thick,rline](2,-4.5)-- (2,0.25);
		\draw[thick,dashed](2,0.25)--(-2,0.25);

		\foreach\x in {1,..., 5}
		{
					\pgfmathtruncatemacro{\nodename}{\x}
		  \node[left]at (-2.2,0.4-\x*.8) {$P{_\x}$};
		  \node(ball\nodename)[draw, circle, inner sep=0.07cm] at (-2,0.4-\x*.8) {}; 
		};
		\node [left] at (-2.2-1, -0.4){deco.$=1$};
		\node [left] at (-2.2-1, -0.4-0.8){deco.$=1$};
		\node [left] at (-2.2-1, -0.4-1.6){deco.$=*$};
		\node [left] at (-2.2-1, -0.4-2.4){deco.$=1$};
		\node [left] at (-2.2-1, -0.4-3.2){deco.$=*$};

  		\foreach\x in {6,..., 10}
		{
							\pgfmathtruncatemacro{\nodename}{\x}
		  \node[left]at (2.7+0.2,+\x*.8-8.8) {$P_\x$};
		  		  \node(ball\nodename)[draw, circle, inner sep=0.07cm] at (2,+\x*.8-8.8) {}; 
		};
		\node[left]at (2.7+0.2+2,+6*.8-8.8){deco.$=*$};
		  \node[left]at (2.7+0.2+2,+7*.8-8.8){deco.$=*$};
		  \node[left]at (2.7+0.2+2,+8*.8-8.8){deco.$=1$};
		  \node[left]at (2.7+0.2+2,+9*.8-8.8){deco.$=*$};
		   \node[left]at (2.7+0.2+2,+10*.8-8.8){deco.$=1$};

              \draw[rline,thick](ball1)--++(0.7,0)--++(0,-1.6)--(ball3);
              \draw[rline,thick](ball2)--++(2,0)--++(0,-0.4)--(ball9);
              \draw[rline,thick](ball4)--++(0.7,0)--++(0,-1.2)--(ball6);
              \draw[rline,thick](ball7)--++(-0.7,0)--++(0,2.4)--(ball10);
              \draw[rline,thick](ball8)--++(-2,0)--++(0,-1.2)--(ball5);

\end{tikzpicture}
\]
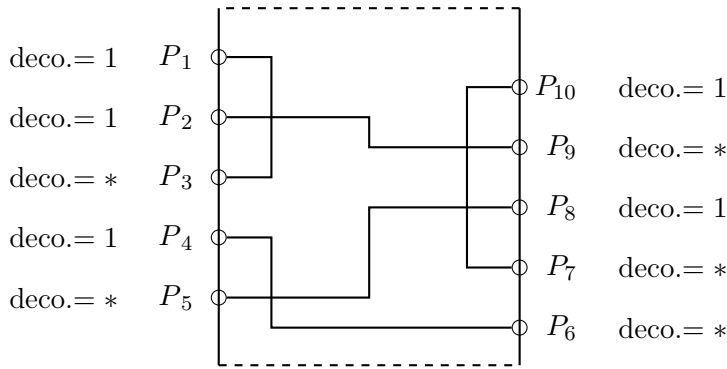
{\bf Figure 7.}  {\em The pair-partition in $\Phi_5^{-1} ( \ee )$ which arises when}

{\em we choose $\xlambda_4 = 3$, $\xlambda_5 = 1$,
$\xlambda_8 = \xlambda_9 = 2$, $\xlambda_{10} = 1$.  }
\end{center}


\vspace{6pt}

The convention for how the parameters $( \xlambda_h )_{h \in \ee^{-1} (*)}$ are 
used in the construction of a $\pi \in \Phi_n^{-1} ( \ee )$ is useful because 
it generates a nice formula for number of crossings, as described in the next 
proposition.
\end{example}

\begin{proposition}  \label{prop:310}
Let $n$ be a positive integer and let $\ee$ be a tuple in $\cD_{(1,*)} (2n)$. 
Consider a tuple $( \xlambda_h )_{h \in \ee^{-1} (*)}$ as discussed in 
Remark \ref{rem:38}, and let $\pi \in \Phi_n^{-1} ( \ee )$ be the pair-partition 
which corresponds to this $( \xlambda_h )_{h \in \ee^{-1} (*)}$, in the way 
described in Example \ref{example:39}.  Then one has
\begin{equation}   \label{eqn:310a}
\emph{\Cr} ( \pi ) = \sum_{h \in \ee^{-1} (*)} ( \xlambda_h - 1 ).
\end{equation}
\end{proposition}

\begin{proof}  Pick a $k \in \ee^{-1} (1)$, and consider the point at height 
$k$ in the picture of $\pi$.  This point is paired with a point at height
$h \in \ee^{-1} (*)$, where $h > k$ (and where we recall that, in our pictures,
heights are measured downwards from the top horizontal side of the rectangle).
Say that the pairing of the points at heights $h$ and $k$ was done in the 
$m$-th step of the construction of $\pi$, where $1 \leq m \leq n$.  This means 
that there were $m-1$ other ``$*$''-decorated points, with heights 
$h_1 < \cdots < h_{m-1} < h$, which were visited and paired before we arrived 
to visit and choose a pair for the point at height $h$.  For some of the $i$'s 
in $\{ 1 , \ldots , m-1 \} $ it may have been the case that the point at height 
$k$ was considered but then 
``skipped'' 
\footnote{ This means that: (i) $h_i > k$; (ii) the parameter $\xlambda_{h_i}$ was 
large enough so that, upon considering the possible choices of pairs for the point 
at height $h_i$, we passed the point at height $k$ before arriving at the choice 
dictated by $\xlambda_{h_i}$.}
in 
the process of choosing the pair for the point at height $h_i$.  We denote by
``$\Skipped (k)$'' the number of values of $1 \leq i \leq m-1$ for which this 
happened.

By starting from how the family of numbers $( \Skipped (k) )_{k \in \ee^{-1} (1)}$
was defined in the preceding paragraph, an elementary counting of crossings
gives the formula
\begin{equation}   \label{eqn:310b}
\Cr ( \pi ) = \sum_{k \in \ee^{-1} (1)} \Skipped (k).
\end{equation}
This is because (as seen by examining the various possible cases) every new 
chord which is drawn in our construction of $\pi$ intersects precisely 
$\Skipped (k)$ of the precedingly drawn chords, where $k$ is the height of 
the ``$1$''-decorated endpoint of the new chord.  When we sum over 
$k \in \ee^{-1} (1)$, we will thus consider all the intersections between
chords in the drawing of $\pi$, with every such intersection counted exactly 
once.

But on the other hand, the sum on the right-hand side of (\ref{eqn:310b}) just
gives the total number of ``skips'' that were made during the construction of 
$\pi$.  The convention (explained at the end of Example \ref{example:39}) for 
how we choose a ``$1$''-decorated point at every step of our construction makes
clear that this total number of skips is equal to 
\begin{equation}   \label{eqn:310c}
\sum_{h \in \ee^{-1} (*)} ( \xlambda_h - 1 )
\end{equation}
(where we now organize our counting according to the ``$*$''-decorated endpoints
of the chords of $\pi$).  By replacing the quantity (\ref{eqn:310c}) on the 
right-hand side of Equation (\ref{eqn:310b}), we obtain the formula for $\Cr ( \pi )$
stated in the proposition.
\end{proof}

$\ $

\section{A two-sided \boldmath{$q$}-Wick formula, 
and proof of Theorem 1.4}

\setcounter{equation}{0}

\subsection{Review of \boldmath{$q$}-creation and 
\boldmath{$q$}-annihilation operators}

$\ $

\noindent
Throughout this subsection we fix a positive integer $d$ and a real 
number $q \in (-1,1)$.  We start from the finite dimensional Hilbert 
space $\bC^d$ and we consider the $q$-deformed Fock space over it,
as defined by Bo$\dot{\text{z}}$ejko and Speicher \cite{BS1991}.  
We will denote this $q$-deformed Fock space by $\cT_{d;q}$.  It is 
described as follows.  

\vspace{6pt}

$\bullet$ First, for every $n \in \bN$ one considers the inner product 
$\langle\cdot  , \cdot\rangle_q$ on $( \bC^d )^{\otimes n}$ which is 
determined by the requirement that
\begin{equation}   \label{eqn:41a}
\langle v_1\otimes \cdots \otimes v_n \, , 
\, w_1\otimes \cdots\otimes w_n\rangle_q \\
= \sum_{\tau \in \cS_n} \langle v_1, w_{\tau (1)} \rangle \cdots 
\langle v_n, w_{\tau (n)}\rangle q^{\iota(\tau)},
\end{equation}   
where $\cS_n$ is the group of permutations of $\{ 1, \ldots , n \}$ and 
where for $\tau \in \cS_n$ we put 
$\iota(\tau )= \ \vline 
\ \{ (i,j) \mid  1\leq i < j \leq n, \tau (i)>\tau (j)\} \ \vline \ .$
The fact that Equation (\ref{eqn:41a}) defines indeed an inner product is
proved in \cite{BS1991}.  

\vspace{6pt}

$\bullet$ $\cT_{d;q}$ is then defined to be the Hilbert space
\begin{equation}   \label{eqn:41b}
\cT_{d;q} := \bC 
\oplus \bigoplus_{n=1}^{\infty} (\bC^d)^{\otimes n}
\ \mbox{ (orthogonal direct sum),}
\end{equation}
where every summand $(\bC^d)^{\otimes n}$ is considered with the inner 
product from (\ref{eqn:41a}).  The number $1 \in \bC$ in the first summand
on the right-hand side of (\ref{eqn:41b}) is called the {\em vacuum vector}
of $\cT_{d;q}$, and will be denoted by $\xivac$.

For every $v \in \bC^d$ one can verify that there exist operators 
$L (v), R(v) \in B ( \cT_{d;q} )$, called 
{\em left $q$-creation operator} and respectively 
{\em right $q$-creation operator} associated to $v$, which are determined 
by the requirements that 
$[ L (v) ] ( \xivac ) = v = [ R (v) ] ( \xivac )$ and that 
\[
\left\{    \begin{array}{l}
{ [ L (v) ] (v_1 \otimes \cdots \otimes v_n)
  = v \otimes v_1 \otimes \cdots \otimes v_n, }        \\
{ [ R (v) ] (v_1 \otimes \cdots \otimes v_n)
  = v_1 \otimes \cdots\otimes v_n \otimes v,  }
\end{array} ,
\ \ \forall \, n \in \bN \mbox{ and } v_1, \ldots , v_n \in \bC^d .
\right.
\]
The adjoints of $L (v)$ and $R(v)$ are called 
{\em left $q$-annihilation operator} and respectively 
{\em right $q$-annihilation operator} associated to $v$.
They will be denoted as $L^{*} (v)$ and $R^{*} (v)$, and they act as follows:
$[ L^{*} (v) ] ( \xivac ) = 0 = [ R^{*} (v) ] ( \xivac )$, and 
\[
\left\{    \begin{array}{l}
{ [ L^{*} (v) ] (v_1 \otimes \cdots \otimes v_n)
  = \sum_{k=1}^n q^{k-1} \langle v_k , v \rangle v_1 \otimes 
  \cdots \otimes v_{k-1} \otimes v_{k+1} \otimes \cdots \otimes v_n,  }    \\
                                                                           \\
{ [ R^{*} (v) ] (v_1 \otimes \cdots \otimes v_n)
  = \sum_{k=1}^n q^{k-1} \langle v_{n-k+1} , v \rangle v_1 \otimes 
  \cdots \otimes v_{n-k} \otimes v_{n-k+2} \otimes \cdots \otimes v_n , }
\end{array}  \right.
\]
for all $n \in \bN$ and $v_1, \ldots , v_n \in \bC^d$. 

The left/right operators of $q$-creation and $q$-annihilation on 
$\cT_{d;q}$ are known to satisfy a number of commutation relations.
Among these, it is of interest for the present paper to record the 
one stated in the following lemma (where we use the standard notation 
$\bigl[ \, X,Y \, \bigr] := XY - YX$ for $X,Y \in B( \cT_{d;q} )$).

\begin{lemma}    \label{lemma:41} 
For every $v,w \in \bC^{d}$ one has
\begin{equation}   \label{eqn:41xa}
\bigl[ \, L (v) + L^{*} (v) , R (w) + R^{*} (w) \, \bigr] =
\bigl( \, \langle w,v \rangle  - \langle v,w \rangle \, \bigr) \, Q, 
\end{equation}
where $Q \in B( \cT_{d;q} )$ is the operator determined by the 
requirement that $Q \xivac = \xivac$ and that
\begin{equation}   \label{eqn:41xb}
Q ( v_1 \otimes \cdots \otimes v_n ) 
= q^n v_1 \otimes \cdots \otimes v_n, 
\ \ \forall \, n \in \bN \mbox{ and } v_1, \ldots , v_n \in \bC^d.
\end{equation}  
\end{lemma}

\begin{proof} By comparing how $L^{*} (v) R (w)$ and 
$R (w) L^{*} (v)$ act on tensors $v_1 \otimes \cdots \otimes v_n$,
one immediately finds that
$\bigl[ \, L^{*} (v)  , R (w) \, \bigr] 
= \langle w,v \rangle \, Q$. 
A similar calculation leads to 
$\bigl[ \, L (v) , R^{*} (w) \, \bigr] 
= - \langle v,w \rangle \, Q$. 
Since it is immediate that 
$\bigl[ \, L (v),  R (w) \, \bigr] 
= \bigl[ \, L^{*} (v),  R^{*} (w) \, \bigr] = 0$, the required 
formula (\ref{eqn:41xa}) follows from the bilinearity of the commutator.
\end{proof}

\begin{corollary}   \label{cor:42}
Let $v,w \in \bC^{d}$ be such that $\langle v,w \rangle \in \bR$.
Then $L (v) + L^{*} (v)$ commutes with $R (w) + R^{*} (w)$ and, 
consequently, the product
$\bigl( \, L(v) + L^{*} (v) \, \bigr) \cdot 
\bigl( \, R (w) + R^{*} (w) \, \bigr)$ is a selfadjoint operator.
\hfill  $\blacksquare$
\end{corollary}

We conclude this review with a notation that will be useful in the 
next subsection: a $q$-annihilation operator is a weighted sum of 
``annihilations at specified distance'' (counting from the left or from 
the right, as needed), and we want to have names for the individual 
terms of this weighted sum.

\begin{notation-and-remark}   \label{def:43}
Let $\cTalg_{d;q}$ be the dense subspace of $\cT_{d;q}$ which is 
obtained by only considering in Equation (\ref{eqn:41b}) the algebraic 
direct sum of the spaces $( \bC^d )^{\otimes n}$.  For $v \in \bC^d$
and $k \in \bN$ we will denote by $A_k^{(*)} (v)$ the linear operator 
on $\cTalg_{d;q}$ determined by the requirements that 
$[A_k^{(*)} (v)] ( \xivac ) = 0$ and that for $n \in \bN$ and 
$v_1, \ldots , v_n \in \bC^d$ we have
\[
[A_k^{(*)} (v)] (v_1 \otimes \cdots \otimes v_n)  
= \left\{  \begin{array}{ll}
0, & \mbox{ if $n < k$};       \\
q^{k-1} \langle v_k , v \rangle v_1 \otimes \cdots \otimes v_{k-1} 
                        \otimes v_{k+1} \otimes \cdots \otimes v_n,
   & \mbox{ if $n \geq k$.}
\end{array}  \right.
\]
The ``$(*)$'' included in the notation is merely a reminder 
that $A_k^{(*)} (v)$ will be viewed as a piece of the annihilation 
operator $L^{*}(v)$ (but we are not trying to identify $A_k^{(*)} (v)$ 
as the adjoint of some other operator).  It is clear that one has
the formula
\begin{equation}   \label{eqn:41y}
[ L^{*} (v) ] (v_1 \otimes \cdots \otimes v_n) 
= \sum_{k=1}^n  [ A_k^{(*)} (v) ] (v_1 \otimes \cdots \otimes v_n),
\end{equation}   
holding for every $n \in \bN$ and $v, v_1, \ldots , v_n \in \bC^d$.

For the sake of uniformity in notations, for every $v \in \bC^d$ we will 
also consider a linear operator $A_1^{(1)} (v)$ on $\cTalg_{d;q}$,
which is defined by
\[
[ A_1^{(1)} (v) ] ( \xi ) = [ L(v) ] ( \xi ), 
\ \ \forall \, \xi \in \cTalg_{d;q}.
\]
The point of this extra notation is that we get to have a 
full collection of linear operators 
\[
A_k^{( \xsigma )} (v), 
\mbox{ defined for $v \in \bC^d$, $\xsigma \in \{ 1, * \}$
       and certain $k \in \bN$ }
\]
(where if $\xsigma = *$ then $k$ can take any value in $\bN$, but if 
$\xsigma =1 $ then we are forced to put $k=1$).

Symmetrically, we consider the collection of linear operators
\[
B_k^{( \xsigma )} (v), 
\mbox{ defined for $v \in \bC^d$, $\xsigma \in \{ 1, * \}$ 
       and certain $k \in \bN$ }
\]
where for $\xsigma =1$ we impose $k=1$ and define $B_1^{(1)} (v)$ 
to be the restriction of $R(v)$ to $\cTalg_{d;q}$, while for 
$\xsigma = *$ and arbitrary $k \in \bN$ we use the formula 
\[
[B_k^{(*)} (v)] (v_1 \otimes \cdots \otimes v_n)  
= \left\{  \begin{array}{ll}
0, & \mbox{ if $n < k$};       \\
q^{k-1} \langle v_{n-k+1} , v \rangle v_1 \otimes \cdots \otimes v_{n-k} 
                     \otimes v_{n-k+2} \otimes \cdots \otimes v_n,
   & \mbox{ if $n \geq k$.}
\end{array}   \right.
\]
The operators $B_k^{(*)} (v)$ are ``pieces of the right $q$-annihilation 
operator associated to $v$'', in the sense that one has
\begin{equation}   \label{eqn:41z}
[ R^{*} (v) ] (v_1 \otimes \cdots \otimes v_n) 
= \sum_{k=1}^n  [ B_k^{(*)} (v) ] (v_1 \otimes \cdots \otimes v_n),
\end{equation}   
holding for every $n \in \bN$ and $v, v_1, \ldots , v_n \in \bC^d$.
\end{notation-and-remark} 

$\ $

\subsection{A two-sided \boldmath{$q$}-Wick formula for 
an alternating left-right monomial}

$\ $

\noindent
In this subsection we fix the following data: a positive integer $d$ and 
a real number $q \in (-1,1)$ (same as in subsection 3.1), and also:

$\bullet$ a positive integer $n$;

$\bullet$ a tuple $\ee = ( \ee (1), \ldots , \ee (2n) ) \in \cD_{(1,*)} (2n)$ 
(with $\cD_{(1,*)} (2n)$ as in Definition \ref{def:34});

$\bullet$ a family of vectors $u_1, \ldots , u_{2n} \in \bC^d$.  

\noindent
We will use the notations related to $q$-creation and $q$-annihilation 
operators on the deformed Fock space $\cT_{d;q}$ that were introduced in the 
preceding subsection, and the various notations and facts pertaining to $\ee$ 
that were discussed in Section 2.  

In reference to the data fixed above, we consider the product of operators
\begin{equation}   \label{eqn:42a}
M := R^{\ee (2n)} (u_{2n}) \cdot L^{\ee (2n-1)} (u_{2n-1}) \cdots 
R^{\ee (2)} (u_2) \cdot L^{\ee (1)} (u_1) \in B( \cT_{d;q} ),
\end{equation}
and we will examine the action of $M$ on the vacuum vector 
$\xivac \in \cT_{d;q}$.  It is quite easy (by keeping in mind 
how $\cD_{(1,*)} (2n)$ is defined) to see that $\xivac$ is an eigenvector 
for $M$; the goal of the present subsection is to describe explicitly 
what is the corresponding eigenvalue.  The eigenvalue will appear (cf. 
Proposition \ref{prop:49} below) as a sum over pair-partitions, giving
a formula of ``Wick'' type.

The fact that $M$ in (\ref{eqn:42a}) has alternating ``$L$'' and ``$R$''
factors is due to the intended use of $M$, towards finding a relation to 
semi-meander polynomials.  It is actually easy to extend the result of 
Proposition \ref{prop:49} to a two-sided $q$-Wick formula for general 
products of operators $L^{\xsigma} (v)$ and $R^{\xsigma} (v)$, with 
$\xsigma \in \{ 1,* \}$ and $v \in \bC^d$ -- see Section 4.1 below.

\begin{notation}   \label{def:44}
Recall that for every $h \in \ee^{-1} (*) \subseteq \{ 2, 3, \ldots, 2n \}$
we defined a choice number $\Choice_{\ee} (h) \in \bN$ (cf. Notation 
\ref{rem:38}).  We will denote
\begin{equation}   \label{eqn:42b}
\xLambda := \Bigl\{ ( \xlambda_1, \ldots , \xlambda_{2n} ) \in \bN^{2n}
\begin{array}{ll}
\vline  &  \mbox{ if $h \in \ee^{-1} (1)$, then $\xlambda_h = 1$;}  \\
\vline  &  \mbox{ if $h \in \ee^{-1} (*)$, then 
                  $1 \leq \xlambda_h \leq \Choice_{\ee} (h)$ }
\end{array}  \Bigr\}  .
\end{equation}
Note that a $(2n)$-tuple in $\xLambda$ is nothing but a tuple of choices
$( \xlambda_h )_{h \in \ee^{-1} (*)}$ as considered in (\ref{eqn:38c}) of 
Remark \ref{rem:38}, where we ``filled in some values of $1$'' by 
putting $\xlambda_k = 1$ for $k \in \ee^{-1} (1)$.  Thus $\xLambda$ is in 
natural bijection with the set of pair-partitions 
$\Phi_n^{-1} ( \ee ) \subseteq \cP_2 (2n)$, in the way discussed in 
Remark \ref{rem:38} and in Example \ref{example:39}.

The role of $\xLambda$ in calculations related to $M ( \xivac )$ shows up 
in the following lemma. 
\end{notation}

\begin{lemma}   \label{lemma:45}
One has
\begin{equation}   \label{eqn:45xc}
M ( \xivac ) = 
\end{equation}
\[
= \sum_{ (\xlambda_1, \ldots , \xlambda_{2n} ) \in \xLambda} 
\ [ B_{\xlambda_{2n}}^{( \ee (2n) )} (u_{2n}) \cdot
A_{\xlambda_{2n-1}}^{( \ee (2n-1) )} ( u_{2n-1} ) \cdots 
B_{\xlambda_2}^{( \ee (2) )} (u_2) \cdot
A_{\xlambda_1}^{( \ee (1) )} (u_1) ]  ( \xivac ),
\]
where the operators of the form $A_k^{( \xsigma )} (v)$ 
and $B_k^{ ( \xsigma )} (v)$ are as in Notation \ref{def:43}.
\end{lemma}

\begin{proof}
Referring to the definition of $M$ in Equation (\ref{eqn:42a}), 
consider the vectors $\xi_0, \xi_1, \ldots , \xi_{2n}$
$\in \cTalg_{d;q}$ obtained by
putting $\xi_0 := \xivac$ and then
\begin{equation}   \label{eqn:45xd}
\left\{   \begin{array}{l}
\xi_1 = [ L^{\ee (1)} (u_1) ] ( \xi_0 ),
\ \xi_2 = [ R^{\ee (2)} (u_2) ] ( \xi_1 ), \ldots ,             \\
                                                                \\
\xi_{2n-1} = [ L^{\ee (2n-1)} (u_{2n-1}) ] ( \xi_{2n-2} ),
\ \xi_{2n} = [ R^{\ee (2n)} (u_{2n}) ] ( \xi_{2n-1} ).
\end{array}  \right.
\end{equation}
Clearly, $M( \xivac ) = \xi_{2n}$.

For every $p \in \bN \cup \{ 0 \}$, let $\cV_p$ denote the copy of 
$( \bC^d )^{\otimes p}$ which sits inside $\cTalg_{d;q}$.  In view of how 
$q$-creation and $q$-annihilation operators act on a space $\cV_p$ (by
mapping it to $\cV_{p + 1}$ and respectively $\cV_{p - 1}$), it is 
immediate that the vectors introduced in (\ref{eqn:45xd}) are such that 
$\xi_1 \in \cV_{p_1}$, 
$\xi_2 \in \cV_{p_2}, \ldots , \xi_{2n} \in \cV_{p_{2n}}$, with 
\begin{equation}   \label{eqn:45xe}
p_h := 
\ \vline \ \{ 1 \leq i \leq h \mid \ee (i) = 1 \} \ \vline \ - 
\ \vline \ \{ 1 \leq i \leq h \mid \ee (i) = * \} \ \vline\ 
\, , \mbox{ for $1 \leq h \leq 2n$}.
\end{equation}

For every $h \in \{ 1, \ldots , 2n \}$ such that $\ee (h) =1$ we have
\begin{equation}   \label{eqn:45xf}
\xi_h = \left\{  \begin{array}{l}
\mbox{ $[ L(u_h) ] ( \xi_{h-1} )$, if $h$ is odd}    \\
\mbox{ $[ R(u_h) ] ( \xi_{h-1} )$, if $h$ is even}
\end{array}  \right.
= \left\{  \begin{array}{l}
\mbox{ $[ A_1^{(1)} (u_h) ] ( \xi_{h-1} )$, if $h$ is odd}     \\
\mbox{ $[ B_1^{(1)} (u_h) ] ( \xi_{h-1} )$, if $h$ is even}.
\end{array}  \right. 
\end{equation}
On the other hand, for every $h \in \{ 1, \ldots , 2n \}$ such that 
$\ee (h) =*$ we have
\begin{equation}   \label{eqn:45xg}
\xi_h = \left\{  \begin{array}{l}
\mbox{ $[ L^{*} (u_h) ] ( \xi_{h-1} )$, if $h$ is odd}     \\
\mbox{ $[ R^{*} (u_h) ] ( \xi_{h-1} )$, if $h$ is even}
\end{array}  \right.
= \left\{  \begin{array}{ll}
\sum_{\xlambda = 1}^{p_{h-1}} [ A_{\xlambda}^{(*)} (u_h) ] ( \xi_{h-1} ), 
                                                   &  \mbox{ if $h$ is odd}     \\
                                                   &                            \\
\sum_{\xlambda = 1}^{p_{h-1}} [ B_{\xlambda}^{(*)} (u_h) ] ( \xi_{h-1} ), 
                                                   &  \mbox{ if $h$ is even.}
\end{array}  \right. 
\end{equation}
In (\ref{eqn:45xg}) we took into account that the action of an operator 
$L^{*} (v)$ on a specified $\cV_p$ can be replaced by the action of 
$\sum_{\xlambda =1}^p A_{\xlambda}^{(*)} (v)$ (cf. Equation (\ref{eqn:41y})
in Remark \ref{def:43}), with a similar statement holding for $R^{*} (v)$ and
$\sum_{\xlambda =1}^p B_{\xlambda}^{(*)} (v)$.

In connection to Equation (\ref{eqn:45xg}) we also make the remark that 
$p_{h-1} = \Choice_{\ee} (h)$ (compare the definition of $p_{h-1}$ to
Equation (\ref{eqn:38a}) in Remark \ref{rem:38}); hence the sums indicated 
in (\ref{eqn:45xg}) have precisely $\Choice_{\ee} (h)$ terms.

When we write the recursion for the vectors $\xi_1, \xi_2, \ldots , \xi_{2n}$
by using the Equations (\ref{eqn:45xf}) and (\ref{eqn:45xg}), we get to have 
$\xi_{2n}$ written precisely as on the right-hand side of 
Equation (\ref{eqn:45xc}).  Since we know that $\xi_{2n} = M( \xivac )$, this 
concludes the proof.
\end{proof}

We now concentrate our attention on a tuple 
$( \xlambda_1, \ldots , \xlambda_{2n} ) \in \xLambda$.
We know that $( \xlambda_h )_{h \in \ee^{-1} (*)}$ is a 
tuple of choices which parametrizes a pair-partition 
$\pi \in \Phi_n^{-1} ( \ee ) \subseteq \cP_2 (2n)$, where  
the map $\Phi_n : \cP_2 (2n) \to \cD_{(1,*)} (2n)$ is as 
described in Notation \ref{def:36}, and where the procedure
for finding $\pi$ is described in detail in Remark \ref{rem:38}
and Example \ref{example:39}.

\begin{lemma}   \label{lemma:46}
With $\xlambda \in \xLambda$ and $\pi \in \Phi_n^{-1} ( \ee )$ as above,
we have
\begin{equation}   \label{eqn:46xa}
[ B_{\xlambda_{2n}}^{( \ee (2n) )} (u_{2n})  \cdot
A_{\xlambda_{2n-1}}^{( \ee (2n-1) )} ( u_{2n-1} )  \cdots 
B_{\xlambda_2}^{( \ee (2) )} (u_2) \cdot
A_{\xlambda_1}^{( \ee (1) )} (u_1) ] ( \xivac ) = c \, \xivac,
\end{equation}
where the scalar $c$ is described as follows:
\begin{equation}   \label{eqn:46xb}
c = \prod_{  \begin{array}{c}
{\scriptstyle \{ k,h \} \ pair \ in \ \Ht_n \cdot \pi} \\ 
{\scriptstyle with \ \ee (h) = *, \ \ee (k) = 1}
\end{array} } \  \Bigl( \, \langle u_k, u_h \rangle \cdot 
                           q^{\xlambda_h - 1} \, \Bigr) .
\end{equation}
\end{lemma}

\begin{proof}  Consider the vectors 
$\eta_0, \eta_1, \ldots , \eta_{2n} \in \cTalg_{d;q}$ obtained by
putting $\eta_0 := \xivac$ and then
\begin{equation}   \label{eqn:46xc}
\left\{   \begin{array}{l}
\eta_1 = [ A_{\xlambda_1}^{\ee (1)} (u_1) ] ( \eta_0 ),
\ \eta_2 = [ B_{\xlambda_2}^{\ee (2)} (u_2) ] ( \eta_1 ), \ldots ,             \\
                                                                              \\
\eta_{2n-1} = [ A_{\xlambda_{2n-1}}^{\ee (2n-1)} (u_{2n-1}) ] ( \eta_{2n-2} ),
\ \eta_{2n} = [ B_{\xlambda_{2n}}^{\ee (2n)} (u_{2n}) ] ( \eta_{2n-1} ).
\end{array}  \right.
\end{equation}
The vector on the left-hand side of Equation (\ref{eqn:46xa}) is then 
$\eta_{2n}$.

We next observe that
$\eta_1 \in \cV_{p_1}, \eta_2 \in \cV_{p_2}, \ldots , 
\eta_{2n} \in \cV_{p_{2n}}$, 
where $p_1, \ldots, p_{2n} \in \bN \cup \{ 0 \}$ are exactly as in 
Equation (\ref{eqn:45xe}) in the proof of the preceding lemma, and the 
spaces ``$\cV_p$'' have the same meaning as in the said proof.  The 
specifics of how the operators $A_{\xlambda}^{(\xsigma)} (v)$ and 
$B_{\xlambda}^{(\xsigma)} (v)$ act on tensors actually give us that 
every $\eta_h$ is of the form
\begin{equation}   \label{eqn:46xd}
\eta_h = c_h \, \zeta_h,
\end{equation}
where $c_h \in \bC$ and $\zeta_h$ is a tensor product of $p_h$ 
vectors 
\footnote{ It may happen (e.g. for $h = 2n$) that $p_h = 0$; in such 
a case, the vector $\zeta_h$ of Equation (\ref{eqn:46xd}) is 
$\zeta_h = \xivac$. } 
picked 
(in some order) out of $u_1, \ldots, u_h$.  We can, moreover, follow on
how $c_h$ and $\zeta_h$ are obtained from $c_{h-1}$ and $\zeta_{h-1}$, 
depending on whether $\ee (h)$ is a ``$1$'' or a ``$*$'':

\noindent
$\bullet$ If $\ee (h) = 1$ then $\zeta_h$ is obtained out of $\zeta_{h-1}$ 
by adding a component (either at the left end or at the right end of the 
tensor), and we have $c_h = c_{h-1}$.

\noindent $\bullet$ If $\ee (h) = *$ then $\zeta_h$ is obtained out of 
$\zeta_{h-1}$ by removing a component ``$u_k$'' (which could be anywhere in 
the tensor), and we have 
\[
c_h = \langle u_k, u_h \rangle \, q^{\xlambda_h - 1} \cdot c_{h-1} .
\]

We leave it as an exercise to the reader to check that the above procedure
of adding and removing components in the tensors $\zeta_h$ corresponds 
precisely to the procedure used in Section 2 in order to obtain the 
pair-partition $\pi \in \Phi_n^{-1} ( \ee )$ which is encoded by the 
choice-numbers $( \xlambda_h )_{h \in \ee^{-1} (*)}$.  In particular, the 
scalar $c_{2n}$ comes out to be a product indexed by pairs of the partition 
$\Ht_n \cdot \pi$, exactly in the way stated on the right-hand side of 
Equation (\ref{eqn:46xb}).  Since $\zeta_{2n} = \xivac$ and $\eta_{2n}$ is 
the vector on the left-hand side of (\ref{eqn:46xa}), this concludes the proof.
\end{proof}

\begin{example}   \label{example:47}
For illustration, we show how the procedure explained in the proof of 
Lemma \ref{lemma:46} works on a concrete example.  Let us consider again 
the situation discussed in Example \ref{example:39}, where $n = 5$, 
$\ee = (1,1,1,*,*,1,1,*,*,*) \in \cD_{(1,*)} (10)$, and the tuple 
$(\xlambda_1, \ldots , \xlambda_{10} ) \in \xLambda$ used for Figure 7 is
$(1,1,1,3,1,1,1,2,2,1)$.  Thus the vector considered in Lemma \ref{lemma:46}
is, in this concrete example:
\[
[ B_1^{(*)} (u_{10}) \, A_2^{(*)} (u_9) \, B_2^{(*)} (u_8) \, A_1^{(1)} (u_7)
\, B_1^{(1)} (u_6) \, A_1^{(*)} (u_5)  \, \times
\]
\begin{equation}   \label{eqn:47xa}
\times \,  B_3^{(*)} (u_4) \, A_1^{(1)} (u_3)
\, B_1^{(1)} (u_2) \, A_1^{(1)} (u_1) ] ( \xivac ).
\end{equation}

The last 3 of the 10 factors in the above product are creation operators, 
and they map $\xivac$ to $c_3 \zeta_3$, where $c_3 =1$ and 
$\zeta_3 = u_3 \otimes u_1 \otimes u_2$.  Let us examine 
what happens when the next 2 factors of the product (both of them doing 
annihilation) are applied.  We have:
\[
[ B_3^{(*)} (u_4) ] ( u_3 \otimes u_1 \otimes u_2 ) = 
\Bigl( \, q^2 \, \langle u_3, u_4 \rangle \, \Bigr) \, u_1 \otimes u_2,
\]
and then
\[
[ A_1^{(*)} (u_5) ] ( u_1 \otimes u_2 ) = 
\Bigl( \, \langle u_1, u_5 \rangle \, \Bigr) \, u_2;
\]
hence $\zeta_5 = u_2$ and 
$c_5 = \langle u_1, u_5 \rangle \ \langle u_3, u_4 \rangle \ q^2$.
Observe that this corresponds exactly (in the respect of what points are 
being paired in the rectangular picture, and also in the respect of ``how 
many points are skipped'' at every step, and from what direction) to the 
construction shown in Figure 6.  In particular, the equality 
$\zeta_5 = u_2$ corresponds to the fact that the point at height 2 is still 
not paired in Figure 6, while the other four points at heights $\leq 5$ are 
paired -- height 3 paired with height 4, then height 1 paired with height 5, 
corresponding to the two inner products multiplied in the formula for $c_5$.

The patient reader can verify that when we apply to $u_3$ the next 3 factors 
from the product (\ref{eqn:47xa}), we arrive to calculate
\[
[ B_2^{(*)} (u_8) ] ( u_7 \otimes u_2 \otimes u_6 )
= \Bigl( \, q \, \langle u_2 , u_8 \rangle \, \Bigr) \, u_7 \otimes u_6,
\]
and this corresponds exactly to performing the next step (choosing a pair 
for the point at height 8) in the construction which led to Figure 7.
Finally, upon applying the remaining 2 factors
$B_1^{(*)} (u_{10}) A_2^{(*)} (u_9)$ to $u_7 \otimes u_6$ we arrive to the 
vector $\eta_{10} = c_{10} \, \zeta_{10}$, where $\zeta_{10} = \xivac$ and 
\[
c_{10} = \langle u_7, u_{10} \rangle
\ \langle u_6, u_9 \rangle \ \langle u_2, u_8 \rangle 
\ \langle u_1, u_5 \rangle \ \langle u_3, u_4 \rangle\, q^4 \ ,
\]
exactly as claimed in Lemma \ref{lemma:46}.
\end{example}

\begin{remark}  \label{rem:48}
In the framework of Lemma \ref{lemma:46}, one has
\[
\prod_{  \begin{array}{c}
{\scriptstyle \{ k,h \} \ pair \ in \ \Ht_n \cdot \pi} \\ 
{\scriptstyle with \ \ee (h) = *, \ \ee (k) = 1}
\end{array} } \ q^{\xlambda_h - 1} 
= \prod_{h \in \ee^{-1} (*)} q^{\xlambda_h -1}
= q^{\Cr ( \pi )},
\]
where at the latter equality sign we invoked Proposition \ref{prop:310}.
Thus the scalar $c$ appearing in Equation (\ref{eqn:46xb}) of Lemma 
\ref{lemma:46} can also be written as
\begin{equation}   \label{eqn:48xa}
c = q^{\Cr ( \pi )} \cdot \prod_{  \begin{array}{c}
{\scriptstyle \{ k,h \} \ pair \ in \ \Ht_n \cdot \pi} \\ 
{\scriptstyle with \ \ee (h) = *, \ \ee (k) = 1}
\end{array} } \  \langle u_k, u_h \rangle .
\end{equation}
\end{remark}

The Wick type formula announced in the title of the subsection is
then stated as follows.

\begin{proposition}   \label{prop:49}
Let $M$ be the operator defined in Equation (\ref{eqn:42a}).  Then
\begin{equation}     \label{eqn:49xa}
M ( \xivac ) = \sum_{\pi \in \Phi_n^{-1} ( \ee )} 
\,  \Bigl[ \, q^{\emph{\Cr} ( \pi )} \cdot 
\prod_{  \begin{array}{c}
{\scriptstyle \{ k,h \} \ pair \ in \ \Ht_n \cdot \pi} \\ 
{\scriptstyle with \ \ee (h) = *, \ \ee (k) = 1}
\end{array} } 
\  \langle u_k, u_h \rangle \, \Bigr] \cdot \xivac .
\end{equation}
\end{proposition}

\begin{proof} This follows by combining Lemmas \ref{lemma:45} and 
\ref{lemma:46}, where we use the natural bijection between $\xLambda$
and $\Phi_n^{-1} ( \ee )$ and the formula (\ref{eqn:48xa}) noticed 
in Remark \ref{rem:48}.
\end{proof}

\begin{remark}    \label{rem:49}
The $q$-Wick formula is usually stated as saying that $\phivac (M)$ is 
equal to the scalar which amplifies $\xivac$ on the right-hand side of 
Equation (\ref{eqn:49xa}).

We note here that, if we focus on $\phivac (M)$ rather than on the vector 
$M( \xivac )$, it is easy to extend the two-sided $q$-Wick formula to the 
framework where the tuple $\ee$ would be picked in $\{ 1,* \}^{2n}$ without assuming 
(as we did throughout this subsection) that $\ee \in \cD_{(1,*)} (2n)$.
Indeed, if $\ee$ was to be picked from 
$\{ 1,* \}^{2n} \setminus \cD_{(1,*)} (2n)$, then it is immediate that 
the vector $M( \xivac )$ would come out orthogonal to $\xivac$, thus giving 
$\phivac (M) = 0$.
\end{remark}

$\ $

\subsection{Alternating monomials in 
\boldmath{$L_{i;q}^{\xsigma}$} and \boldmath{$R_{i;q}^{\xsigma}$}}

$\ $

\noindent
In this subsection we continue to keep fixed the $d \in \bN$, $q \in (-1,1)$,
$n \in \bN$ and $\ee \in \cD_{(1,*)} (2n)$ from the preceding subsection,
but not the vectors $u_1, \ldots , u_{2n} \in \bC^d$ that were used there.  
We will do some further processing of the formula obtained in Proposition 
\ref{prop:49}, in the special case when $u_1, \ldots , u_{2n}$ are picked from 
the standard orthonormal basis of $\bC^d$.

So for $1 \leq i \leq d$, let $e_i := (0, \ldots, 0,1,0, \ldots , 0)$ with the entry 
of $1$ on position $i$, and let us put, consistently with the notations used 
in the Introduction,
\[
L_{i;q} := L( e_i ), \ \ R_{i;q} := R( e_i ).
\]
Consider a $(2n)$-tuple $I : \{ 1, \ldots , 2n \} \to \{ 1, \ldots , d \}$ 
and let us examine the $q$-Wick formula from 
Proposition \ref{prop:49} in the case when $u_1 = e_{I(1)}, \ldots , u_{2n} = e_{I(2n)}$.
Each of the inner products $\langle u_k , u_h \rangle$ appearing in the formula is 
equal to $0$ or to $1$, and asking that all these inner products are equal to $1$ amounts
to a compatibility condition between $I$ and $\pi$, which we record in the following 
(ad-hoc) definition.

\begin{definition}   \label{def:410}
We will say that a $(2n)$-tuple $I : \{ 1, \ldots , 2n \} \to \{ 1, \ldots , d \}$ 
and a pair-partition $\pi \in \cP_2 (2n)$ are {\em height-compatible} to mean 
that the following condition is fulfilled:
\begin{equation}   \label{eqn:410xa}
I(h) = I(k), \ \mbox{ whenever $\{ h,k \}$ is a pair of $\h \cdot \pi$. }
\end{equation}
\end{definition}

\begin{remark}   \label{rem:411}
Note that Definition \ref{def:410} does not involve the tuple 
$\ee \in \cD_{(1,*)} (2n)$, it just prescribes a relation between 
$I$ and $\pi$.  If we fix a $\pi \in \cP_2 (2n)$, then height-compatibility 
amounts to a system of $n$ equations that have to be satisfied by 
a $(2n)$-tuple $I$ (it is clear, in particular, that there are precisely 
$d^n$ tuples $I$ which are height-compatible 
to 
\footnote{ But things are not so nice if, vice-versa, an $I$ is given and 
we count the pair-partitions $\pi \in \cP_2 (2n)$ which are height-compatible 
to it -- the answer can vary between ``$0$'' (e.g. when 
$I(1) = I(2) = \cdots = I(2n-1) \neq I(2n)$) and ``$(2n-1)!!$''
(obtained when $I(1) = I(2) = \cdots = I(2n)$).}
any 
given $\pi \in \cP_2 (2n)$).

For illustration, say for instance that $n=5$ and $\pi \in \cP_2 (10)$
is as depicted in Figure 4.  Then an 
$I : \{ 1, \ldots , 10 \} \to \{ 1, \ldots , d \}$ is height-compatible 
with this $\pi$ if and only if it satisfies 
\[
I(1) = I(4), \, I(2) = I(5), \, I(3) = I(8), \, I(6) = I(10),
\, I(7) = I(9).
\]
\end{remark}

When using the terminology from Definition \ref{def:410}, the special case of 
the two-sided $q$-Wick formula that we are interested in takes the following
form.

\begin{corollary}  \label{cor:412}
Consider the data $d \in \bN$, $q\in (-1,1)$, $n \in \bN$ and 
$\ee \in \cD_{(1,*)} (2n)$ fixed in this subsection.  For every 
$(2n)$-tuple $I : \{ 1, \ldots , 2n \} \to \{1, \ldots , d \}$
we have
\begin{equation}   \label{eqn:412xa}
\phivac \Bigl( \,  R_{I(2n);q}^{\ee (2n)} L_{I(2n-1);q}^{\ee (2n-1)} \cdots
R_{I(2);q}^{\ee (2)}  L_{I(1);q}^{\ee (1)} \, \Bigr) 
= \sum_{\substack{ \pi\in \Phi_n^{-1}(\ee),  \\ 
\text{height-compatible with}\,I}} 
q^{\emph{\Cr} (\pi)} .
\end{equation}
(The index set in the summation on the right-hand side of 
(\ref{eqn:412xa}) may be empty, in which case the corresponding 
sum is equal to $0$.)
\end{corollary}

\begin{proof}  As an obvious consequence of Proposition \ref{prop:49}
and of how the definition of height-compatibility was made, we get
\begin{equation}   \label{eqn:412xb}
[ R_{I(2n);q}^{\ee (2n)} L_{I(2n-1);q}^{\ee (2n-1)} \cdots
R_{I(2);q}^{\ee (2)}  L_{I(1);q}^{\ee (1)} ] ( \xivac ) 
 = \Bigl[ \sum_{\substack{ \pi\in \Phi_n^{-1}(\ee),  \\ 
\text{height-compatible with}\,I}} 
q^{\Cr (\pi)} \Bigr] \cdot \xivac ,
\end{equation}
which in turn implies the formula stated in the corollary.
\end{proof}

$\ $

\subsection{Proof of Theorem 1.4.}
 
$\ $

\noindent
In this subsection we fix $d \in \bN$ and $q \in (-1,1)$, 
and we will use the various notations pertaining to this 
$d$ and $q$ that were considered in Sections 3.1-3.3. 
Our goal is to prove Theorem 1.4 from the Introduction.

The theorem refers to the polynomials 
$\widetilde{\xR}_n (t,u)$, which were introduced in Notation 1.3.  
Since the definition of $\widetilde{\xR}_n$ relies on the concepts of
``number of crossings'' and of ``number of closed curves'' for a 
self-intersecting meandric system, we start by indicating the precise 
formulas for these numbers.  So let $n$ be a positive integer, let 
$\pi , \sigma$ be in $\cP_2 (2n)$, and consider the self-intersecting 
meandric system which is obtained by drawing $\pi$ above a horizontal 
line and $\sigma$ under it (as in Figure 2 in the Introduction).  Then:

\vspace{6pt}

\noindent
$\bullet$ the {\em number of crossings} of the said self-intersecting 
meandric system is $\Cr ( \pi ) + \Cr ( \sigma )$, and

\vspace{6pt}

\noindent
$\bullet$ its {\em number of closed curves} is defined as 
$| \pi \vee \sigma |$ (in words: the number of blocks of the join 
of $\pi$ and $\sigma$ in the lattice $\cP (2n)$), where the join 
operation ``$\vee$'' and a few other useful terms related to partitions
are reviewed in the next definition.

\vspace{6pt}

\begin{definition-and-remark}   \label{def:413}
Let $m$ be a positive integer.

\vspace{6pt}

(1) We denote by $\cP(m)$ the set of all {\em partitions} of 
$\{ 1, \ldots , m \}$.  A $\pi \in \cP(m)$ is thus of the form 
$\pi=\{ V_1, \ldots, V_k \}$, where the sets $V_1, \ldots, V_k$ (called 
{\em blocks} of $\pi$) are non-empty sets such that 
$\cup_{i=1}^k V_i = \{ 1, \ldots , m \}$ and
$V_i \cap V_j = \emptyset \mbox{ for } i \neq j$.

\vspace{6pt}

(2) Let $\pi$ be in $\cP (m)$.  We will use the notation 
$| \pi |$ for the number of blocks of $\pi$.
We will also use the notation ``$\ecpi$'' 
for the equivalence relation determined by $\pi$ on $\{ 1, \ldots , m \}$; 
that is, for $a,b \in \{ 1, \ldots , m \}$ we will write $a \ecpi b$ to 
mean that $a$ and $b$ belong to the same block of $\pi$. 

\vspace{6pt}

(3) On $\cP (m)$ we will use the partial order given by 
{\em reverse refinement}, where $\pi \leq \pi '$ means by 
definition that every block of $\pi '$ can be written as a 
union of blocks of $\pi$.  

\vspace{6pt}

(4) Any two partitions $\pi , \sigma \in \cP (m)$ have a least 
common upper bound, denoted as $\pi \vee \sigma$, with respect to 
the reverse refinement order.  The partition
$\pi \vee \sigma$ is called the {\em join} of $\pi$ and $\sigma$.  
It is easily verified that its blocks can be explicitly described 
as follows: two numbers $a,b \in \{ 1, \ldots , m \}$ belong to 
the same block of $\pi \vee \sigma$ if and only if there exist 
$k \in \bN$ and 
$a_0, a_1, \ldots , a_{2k} \in \{ 1, \ldots , m \}$ such that 
$a = a_0 \ecpi a_1 \ecsigma a_2 \ecpi \cdots \ecpi a_{2k-1}
\ecsigma a_{2k} = b$.

\vspace{6pt}

(5) Let $s$ be a permutation of $\{ 1, \ldots , m \}$, and consider the 
natural action of $s$ on $\cP (m)$ (if $\pi = \{ V_1, \ldots , V_k \}$,
then $s \cdot \pi = \{ s(V_1), \ldots , s(V_k) \}$.  This action 
respects the join operation, that is:
\begin{equation}   \label{eqn:413b}
s \cdot ( \pi \vee \sigma ) = 
( s \cdot \pi ) \vee ( s \cdot \sigma ),
\ \ \forall \, \pi , \sigma \in \cP (m).
\end{equation} 
\end{definition-and-remark}

\begin{remark}   \label{rem:414}
Let $n$ be in $\bN$ and consider the set of self-intersecting meandric 
systems $\widetilde{\cR}_n$ from Notation 1.3.  Every such system is 
parametrized by a $\pi \in \cP_2 (2n)$ (namely the $\pi$ which appears, in 
the drawing of the system, above the horizontal line), and has:

$\bullet$ number of crossings equal to $\Cr ( \pi )$ (since the rainbow 
pair-partition has $\Cr ( \rho_{2n} ) = 0$);

$\bullet$ number of closed curves equal to $| \pi \vee \rho_{2n} |$.

\noindent
It is immediately seen that, upon plugging these numbers into the Equation 
(\ref{eqn:intro8}) which was used to define the polynomial 
$\widetilde{\xR}_n (t,u)$ in Notation 1.3, one comes to the formula 
\begin{equation}   \label{eqn:414xa}
\widetilde{\xR}_n (t,u) = \sum_{\pi \in \cP_2 (2n)}
t^{|\pi \vee \rho_{2n} |} \, u^{\Cr (\pi)} .
\end{equation}
\end{remark}

\vspace{6pt}

We now proceed to the actual proof of Theorem 1.4.  We will use the 
following lemma.

\begin{lemma}   \label{lemma:415}
Let $n$ be in $\bN$ and let $\pi$ be a pair-partition in $\cP_2 (2n)$.
Consider the set of $(2n)$-tuples
\begin{equation}   \label{eqn:415a}
\cI_{\pi} :=
\Bigl\{ I : \{ 1, \ldots , 2n \} \to \{ 1, \ldots , d \}
\begin{array}{ll}
\vline  &
I(2k-1) = I(2k), \ \ \forall \, 1 \leq k \leq n,  \\
\vline  &
\mbox{and $I$ is height-compatible with $\pi$}
\end{array}   \Bigr\} ,
\end{equation}
where the notion of ``height-compatible'' is as in 
Definition \ref{def:410}.  Then
$\mid \cI_{\pi} \mid = d^{ | \pi \vee \rho_{2n} | }$.
\end{lemma}

\begin{proof} The first requirement imposed on $I$ on the right-hand 
side of (\ref{eqn:415a}) can be expressed as an inequality with respect 
to the reverse refinement order on $\cP (2n)$, namely
\[
Ker (I) \geq \bigl\{ \, \{ 1,2 \}, \ldots , \{ 2n-1, 2n \} \, \bigr\} ,
\]
where $\Ker (I)$ denotes the partition of $\{ 1, \ldots , 2n \}$ into 
level-sets of $I$ (two numbers $a,b \in \{ 1, \ldots , 2n \}$ belong to 
the same block of $\Ker (I)$ if and only if $I(a) = I(b)$).  Upon reviewing
the definition of height-compatibility, we see that the second requirement 
imposed on $I$ in (\ref{eqn:415a}) amounts to the inequality 
$Ker (I) \geq \h \cdot \pi$.  But then, in view of how the join operation is
defined on $\cP (2n)$, we find that the set $\cI_{\pi}$ defined in 
(\ref{eqn:415a}) can be written as
\[
\cI_{\pi} = 
\Bigl\{ I : \{ 1, \ldots , 2n \} \to \{ 1, \ldots , d \} \mid 
Ker (I) \geq \bigl\{ \, \{ 1,2 \}, \ldots , \{ 2n-1, 2n \} \, \bigr\} 
\vee \h ( \pi ) \Bigr\} .
\]
This can be continued with
\[
\begin{array}{ll}
= & 
\bigl\{ I : \{ 1, \ldots , 2n \} \to \{ 1, \ldots , d \} \mid 
Ker (I) \geq ( \h \cdot \rho_{2n} ) \vee ( \h \cdot \pi ) \bigr\}                  \\
  & \mbox{ (by Equation (\ref{eqn:33b}) in Notation \ref{def:33}) }    \\
= &
\bigl\{ I : \{ 1, \ldots , 2n \} \to \{ 1, \ldots , d \} \mid 
Ker (I) \geq \h \cdot ( \rho_{2n} \vee \pi ) \bigr\}                     
\ \ \mbox{ (by Eqn.(\ref{eqn:413b})). } 
\end{array}
\]
In the latter formulation, it becomes clear that all the $(2n)$-tuples 
$I \in \cI_{\pi}$ are obtained, without repetitions, when we prescribe
at will one value $b \in \{ 1, \ldots , d \}$ for every block of 
$\h ( \rho_{2n} \vee \pi )$, and assign $I$ to be identically equal to $b$ 
on that block.  This implies 
$| \cI_{\pi} |  = d^{ | \h \cdot ( \rho_{2n} \vee \pi ) | }$, and since
$| \h \cdot ( \rho_{2n} \vee \pi ) | = | \rho_{2n} \vee \pi |$, the required 
formula for $| \cI_{\pi} |$ follows.
\end{proof}

\vspace{10pt}

{\em Proof of Theorem 1.4.}  The operator $T_{d;q}$ is selfadjoint by Corollary 
\ref{cor:42}.  For the calculation of moments, we must prove (by taking into account 
the above formula Equation (\ref{eqn:414xa})) that
\begin{equation}   \label{eqn:413yz}
\phivac ( T_{d;q}^n ) = \sum_{\pi \in \cP_2 (2n)}
d^{|\pi \vee \rho_{2n} |} \, q^{\Cr (\pi)} ,
\ \ \forall \, n \in \bN .
\end{equation}
For the remaining part of the proof we fix an $n \in \bN$ for which we 
will verify that (\ref{eqn:413yz}) holds.

By taking the power $n$ of the sum which defines $T_{d;q}$, we get
\[
T_{d;q}^n = 
\Bigl( \, \sum_{i=1}^d (L_{i;q}+L_{i;q}^*)(R_{i;q}+R_{i;q}^*) \, \Bigr)^n =
\Bigl( \, \sum_{i=1}^d (R_{i;q}+R_{i;q}^*)(L_{i;q}+L_{i;q}^*) \, \Bigr)^n
\]
\begin{equation}   \label{eqn:413ya}
= 
\sum_{\substack{ I : \{ 1, \ldots , 2n \} \to \{ 1, \ldots , d \}      \\ 
                   \text{with}\,I(2k-1) = I(2k), \, 1 \leq k \leq n } } 
\,  \sum_{\ee: \{ 1, \ldots , 2n \} \to \{ 1,* \} }
\,  R_{I(2n);q}^{\ee (2n)} L_{I(2n-1);q}^{\ee (2n-1)} \cdots
R_{I(2);q}^{\ee (2)}  L_{I(1);q}^{\ee (1)} .
\end{equation}
Upon applying $\phivac$ to (\ref{eqn:413ya}) and then invoking an 
orthogonality observation analogous to the one in the last paragraph of 
Remark \ref{rem:49}, we find that 
\begin{equation}   \label{eqn:413yb}
\phivac ( T_{d;q}^n ) = 
\sum_{\substack{ I : \{ 1, \ldots , 2n \} \to \{ 1, \ldots , d \}      \\ 
                   \text{with}\,I(2k-1) = I(2k), \, 1 \leq k \leq n } } 
\Bigl[ \,  \sum_{\ee \in \cD_{(1,*)} (2n)}
\phivac( \,  R_{I(2n);q}^{\ee (2n)} L_{I(2n-1);q}^{\ee (2n-1)} \cdots
R_{I(2);q}^{\ee (2)}  L_{I(1);q}^{\ee (1)} \, ) \, \Bigr] .
\end{equation}

We next note that for any fixed 
$I : \{ 1, \ldots , 2n \} \to \{ 1, \ldots , d \}$ we can write:
\[
\sum_{\ee \in \cD_{(1,*)} (2n)}
\phivac( \, R_{I(2n);q}^{\ee (2n)} L_{I(2n-1);q}^{\ee (2n-1)} \cdots
R_{I(2);q}^{\ee (2)}  L_{I(1);q}^{\ee (1)} \, )
\]
\[
= \sum_{\ee \in \cD_{(1,*)} (2n)} 
\Bigl[ \, 
\sum_{\substack{ \pi \in \Phi_n^{-1} ( \ee ),           \\ 
                 \text{height-compatible with}\, I} } 
q^{\Cr ( \pi )} \, \Bigr]
\ \ \mbox{ (by Corollary \ref{cor:412})}
\]
\[
= \sum_{\substack{ \pi \in \cP_2 (2n),            \\ 
                     \text{height-compatible with}\, I} } 
\ q^{\Cr ( \pi )} ,
\]
where at the latter equality sign we took into account that 
$\cup_{\ee \in \cD_{(1,*)} (2n)} \Phi_n^{-1} ( \ee ) = \cP_2 (2n)$,
disjoint union.

When the result of the calculation from the preceding paragraph is 
plugged into Equation (\ref{eqn:413yb}), we find that
\begin{equation}   \label{eqn:413yc}
\phivac ( T_{d;q}^n ) = 
\sum_{\substack{ I : \{ 1, \ldots , 2n \} \to \{ 1, \ldots , d \}      \\ 
                   \text{with}\,I(2k-1) = I(2k), \, 1 \leq k \leq n } } 
\Bigl[ \, \sum_{\substack{ \pi \in \cP_2 (2n),            \\ 
                     \text{height-compatible with}\, I} } 
q^{\Cr ( \pi )} \, \Bigr] .
\end{equation}

Finally, we change the order of summation on the right-hand side of Equation 
(\ref{eqn:413yc}).  This leads to
\[
\phivac ( T_{d;q}^n ) 
= \sum_{\pi \in \cP_2 (2n)} q^{\Cr ( \pi )} \cdot 
| \, \cI_{\pi} \, |,
\]
where the set $\cI_{\pi} \subseteq \{ 1, \ldots , d \}^{2n}$ is precisely the
one considered in Lemma \ref{lemma:415}.  The required formula for 
$\phivac ( T_{d;q}^n )$ follows, since Lemma \ref{lemma:415} has 
established that $| \, \cI_{\pi} \, | = d^{ | \pi \vee \rho_{2n} | }$.
\hfill $\blacksquare$

$\ $

\section{Some miscellaneous remarks}
\setcounter{equation}{0}

\subsection{A more general two-sided \boldmath{$q$}-Wick formula}

$\ $

\noindent
The derivation of two-sided $q$-Wick formula in Section 3.2 focused 
on the case needed in our main theorem, where the considered product of 
operators of the form $L^{\xsigma} (u)$, $R^{\xsigma} (u)$ has $2n$ 
alternating factors ``$L$'' and ``$R$''.  We will outline here how this 
formula generalizes to products that are not necessarily alternating.

So let us consider again the data that was fixed throughout Section 3.2:
$d, n \in \bN$, $q \in (-1,1)$, a tuple 
$\ee = ( \ee (1), \ldots , \ee (2n) ) \in \cD_{(1,*)} (2n)$ and some
vectors $u_1, \ldots , u_{2n} \in \bC^d$.  In addition to that, we now 
fix a tuple
\begin{equation}   \label{eqn:51a}
\chi = ( \chi (1), \ldots , \chi (2n) ) \in \{ \ell , r \}^{2n},
\end{equation}
and we look at the operator
\begin{equation}   \label{eqn:51b}
M := S_{\chi (2n)}^{\ee (2n)} (u_{2n}) \cdots 
S_{\chi (1)}^{\ee (1)} (u_1)  \in B( \cT_{d;q} ),
\end{equation}
where for $u \in \cT_{d;q}$ we denote (with ``$S$'' as a reminder of 
``shift operator''):
\[
L(u) =: S_{\ell} (u), \ \  R(u) =: S_{r} (u). 
\]
Same as in the special case discussed in Section 3.2 (which corresponds to 
putting $\chi = ( \ell, r, \ell, r, \ldots , \ell, r )$), we have that the 
vacuum-vector $\xivac$ is an eigenvector for $M$, and the Wick formula
expresses the corresponding eigenvalue as a summation over $\cP_2 (2n)$.
When writing the derivation of this formula, it turns out that the only 
thing that needs to be changed (in the considerations leading to Proposition
\ref{prop:49} from Section 3.2) is the definition of the ``labels-to-heights'' 
permutation $\Ht_n$ introduced in Notation \ref{def:33}.  What is needed in 
the place of $\Ht_n$ is a permutation ``$\Ht_{\chi}$'' (first noticed in 
\cite{MN2015}, and then put to intensive use in \cite{CNS2015, CNS2015b}) 
which plays an essential role in the combinatorics of two-faced free probability.

\begin{notation}   \label{def:51}
Let $\chi$ be as in (\ref{eqn:51a}), and let us write explicitly:
\[
\left\{ 
\begin{array}{ll}
\chi^{-1} ( \ell ) = \{ i_1, \ldots , i_p \}, 
& \mbox{ with $i_1 < \cdots < i_p$, and}           \\
\chi^{-1} (r) = \{ j_1, \ldots , j_q \}, 
& \mbox{ with $j_1 < \cdots < j_q$, where $q = n-p$. }    
\end{array}  \right.
\]
We denote by $\Ht_{\chi}$ the permutation of $\{ 1, \ldots , 2n \}$ 
defined by
\begin{equation}   \label{eqn:51xa}
\Ht_{\chi} := \left(  \begin{array}{cccccccc}
1   & 2    & \cdots & p   & p+1 & \cdots & 2n-1  & 2n     \\
i_1 & i_2  & \cdots & i_p & j_q & \cdots & j_2   & j_1 
\end{array}  \right) .
\end{equation}
\end{notation}

\begin{example}   \label{example:52}
Suppose that $n=5$ and that 
$\chi = ( r, \ell, \ell, r, \ell, \ell, r,r, \ell, \ell )
\in \{ \ell , r \}^{10}$.  Thus  
$\chi^{-1}(\ell)=\{2, 3, 5, 6, 9, 10 \}$,
$\chi^{-1}(r)=\{1,4,7,8 \}$ and the labels-to-heights permutation
$\Ht_{\chi}$ comes out as 
\[
\Ht_{\chi} := \left(  \begin{array}{cccccccccc}
1 & 2 & 3 & 4 & 5 & 6  & 7 & 8 & 9 & 10     \\
2 & 3 & 5 & 6 & 9 & 10 & 8 & 7 & 4 & 1 
\end{array}  \right) .
\]
When working with this $\chi$, the pictures showing 10 points around 
a rectangle which were used for illustration in Section 2 must now be
adjusted as shown in Figure 8.

\begin{center}
\[
\begin{tikzpicture}[baseline]
		\draw[thick, rline] (-2,0.25) -- (-2, -4.5);
		\draw[thick, dashed](-2,-4.5) -- (2,-4.5) ;
		\draw[thick,rline](2,-4.5)-- (2,0.25);
		\draw[thick,dashed](2,0.25)--(-2,0.25);

		\foreach\x in {1,2}
		{
					\pgfmathtruncatemacro{\nodename}{\x}
		  \node[left]at (-2.2,-0.4*\x-0.4) {$P{_\x}$};
		  \node(ball\nodename)[draw, circle, inner sep=0.07cm] at (-2,-0.4*\x-0.4) {}; 
		};
		\pgfmathtruncatemacro{\nodename}{3}
		  \node[left]at (-2.2, -1.6-0.4*1) {$P{_3}$};
		  \node(ball\nodename)[draw, circle, inner sep=0.07cm] at (-2,-1.6-0.4*1) {}; 
		  		\pgfmathtruncatemacro{\nodename}{4}
		  \node[left]at (-2.2, -1.6-0.4*2) {$P{_4}$};
		  \node(ball\nodename)[draw, circle, inner sep=0.07cm] at (-2,-1.6-0.4*2) {}; 
		  		\pgfmathtruncatemacro{\nodename}{5}
		  \node[left]at (-2.2, -3.6) {$P{_5}$};
		  \node(ball\nodename)[draw, circle, inner sep=0.07cm] at (-2,-3.6) {}; 
		  		  		\pgfmathtruncatemacro{\nodename}{6}
		  \node[left]at (-2.2, -4) {$P{_6}$};
		  \node(ball\nodename)[draw, circle, inner sep=0.07cm] at (-2,-4) {}; 
		  
		  		  		  		\pgfmathtruncatemacro{\nodename}{7}
		  \node[left]at (2.9, -3.2) {$P{_7}$};
		  \node(ball\nodename)[draw, circle, inner sep=0.07cm] at (2,-3.2) {}; 
		  
		  		  		  		\pgfmathtruncatemacro{\nodename}{8}
		  \node[left]at (2.9, -2.8) {$P{_8}$};
		  \node(ball\nodename)[draw, circle, inner sep=0.07cm] at (2,-2.8) {}; 
		  		  		  		  		\pgfmathtruncatemacro{\nodename}{9}
		  \node[left]at (2.9, -1.6) {$P{_9}$};
		  \node(ball\nodename)[draw, circle, inner sep=0.07cm] at (2,-1.6) {}; 
		  
		  		  		  		\pgfmathtruncatemacro{\nodename}{10}
		  \node[left]at (3.1, -1.2+0.8) {$P{_{10}}$};
		  \node(ball\nodename)[draw, circle, inner sep=0.07cm] at (2,-1.2+0.8) {};

\end{tikzpicture}
\]
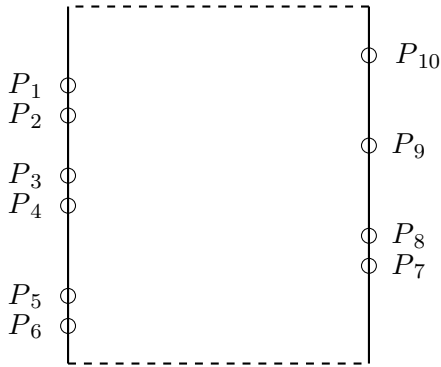
{\bf Figure 8.}  {\em Repositioning of the points $P_1, \ldots , P_{10}$, in 
order to}
 
{\em work with $\chi = ( r, \ell, \ell, r, \ell, \ell, r,r, \ell, \ell )
\in \{ \ell , r \}^{10}$.  } 
\end{center}

\end{example}

\begin{remark}    \label{rem:52}
Upon re-examining the Section 2, one sees that all the constructions and arguments 
presented there were based on the permutation $\Ht_n$ introduced in Notation 
\ref{def:33}, and can be adjusted word-by-word to the framework where we use 
our arbitrary (but fixed) tuple $\chi \in \{ \ell , r \}^{2n}$ by simply replacing 
everywhere ``$\Ht_n$'' by ``$\Ht_{\chi}$''.  First, we get a bijection
\[
\Phi_{\chi} : \cP_2 (2n) \to \cD_{(1,*)} (2n),
\]
defined exactly as in Notation \ref{def:36}.  Then the whole discussion from
Section 2.3 goes through to the $\chi$-framework, and concludes with the analogue 
of Proposition \ref{prop:310}, expressing the number of crossings of a partition 
$\pi \in \Phi_{\chi}^{-1} ( \ee )$ in terms of the tuple of choice-numbers which 
encodes that partition.  A concrete example of how this goes 
is shown in Figure 9.

\begin{center}
\[
\begin{tikzpicture}[baseline]
		\draw[thick, rline] (-2,0.25) -- (-2, -4.5);
		\draw[thick, dashed](-2,-4.5) -- (2,-4.5) ;
		\draw[thick,rline](2,-4.5)-- (2,0.25);
		\draw[thick,dashed](2,0.25)--(-2,0.25);

		\foreach\x in {1,2}
		{
					\pgfmathtruncatemacro{\nodename}{\x}
		  \node[left]at (-2.2,-0.4*\x-0.4) {$P{_\x}$};
		  \node(ball\nodename)[draw, circle, inner sep=0.07cm] at (-2,-0.4*\x-0.4) {}; 
		};
		\pgfmathtruncatemacro{\nodename}{3}
		  \node[left]at (-2.2, -1.6-0.4*1) {$P{_3}$};
		  \node(ball\nodename)[draw, circle, inner sep=0.07cm] at (-2,-1.6-0.4*1) {}; 
		  		\pgfmathtruncatemacro{\nodename}{4}
		  \node[left]at (-2.2, -1.6-0.4*2) {$P{_4}$};
		  \node(ball\nodename)[draw, circle, inner sep=0.07cm] at (-2,-1.6-0.4*2) {}; 
		  		\pgfmathtruncatemacro{\nodename}{5}
		  \node[left]at (-2.2, -3.6) {$P{_5}$};
		  \node(ball\nodename)[draw, circle, inner sep=0.07cm] at (-2,-3.6) {}; 
		  		  		\pgfmathtruncatemacro{\nodename}{6}
		  \node[left]at (-2.2, -4) {$P{_6}$};
		  \node(ball\nodename)[draw, circle, inner sep=0.07cm] at (-2,-4) {}; 
		  
		  		  		  		\pgfmathtruncatemacro{\nodename}{7}
		  \node[left]at (2.9, -3.2) {$P{_7}$};
		  \node(ball\nodename)[draw, circle, inner sep=0.07cm] at (2,-3.2) {}; 
		  
		  		  		  		\pgfmathtruncatemacro{\nodename}{8}
		  \node[left]at (2.9, -2.8) {$P{_8}$};
		  \node(ball\nodename)[draw, circle, inner sep=0.07cm] at (2,-2.8) {}; 
		  		  		  		  		\pgfmathtruncatemacro{\nodename}{9}
		  \node[left]at (2.9, -1.6) {$P{_9}$};
		  \node(ball\nodename)[draw, circle, inner sep=0.07cm] at (2,-1.6) {}; 
		  
		  		  		  		\pgfmathtruncatemacro{\nodename}{10}
		  \node[left]at (3.1, -1.2+0.8) {$P{_{10}}$};
		  \node(ball\nodename)[draw, circle, inner sep=0.07cm] at (2,-1.2+0.8) {};

		  \node [left] at (-2.2-1, -1.2){deco.$=1$};
		  \node [left] at (-2.2-1, -0.8){deco.$=1$};
		  \node [left] at (-2.2-1, -2){deco.$=1$};
		  \node [left] at (-2.2-1, -2.4){deco.$=*$};
		  \node [left] at (-2.2-1, -3.6){deco.$=*$};
		  \node [left] at (-2.2-1, -4){deco.$=*$};
		  
		  \node[left]at (2.7+0.2+2,-3.2){deco.$=1$};
		  \node[left]at (2.7+0.2+2,-2.8){deco.$=*$};
		  \node[left]at (2.7+0.2+2,-1.6){deco.$=*$};
		  		  \node[left]at (2.7+0.2+2,-1.2+0.8){deco.$=1$};

		  \draw[rline,thick](ball2)--++(1.3,0)--++(0,-1.2)--(ball4);
		  
		  		  \draw[rline,thick](ball3)--++(0.7,0)--++(0,-2)--(ball6);
				  \draw[rline, thick](ball5)--++(2,0)--++(0,0.4)--(ball7);
				   \draw[rline, thick](ball8)--++(-0.7,0)--++(0,2.4)--(ball10);
				    \draw[rline, thick](ball1)--++(2,0)--++(0,-0.8)--(ball9);

\end{tikzpicture}
\]
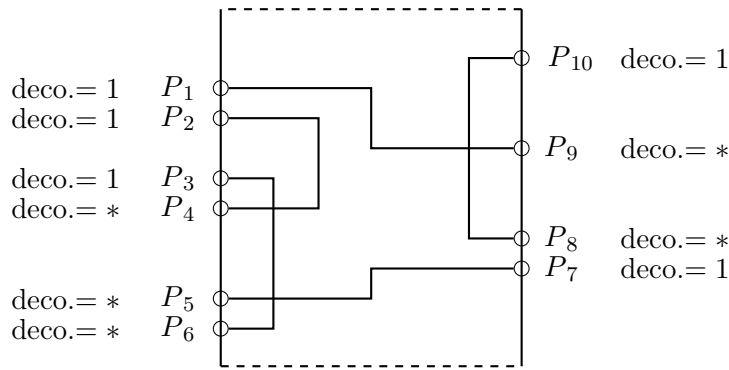
{\bf Figure 9.}  {\em Suppose that $n=5$, 
$\chi =( r, \ell, \ell, r, \ell, \ell, r, r, \ell, \ell )$, 
$\ee =(1, 1, 1, *, 1, *, *, 1, *, *)$.}

{\em The pair-partition 
$\pi = \{  \{1,9\}, \{2,4 \}, \{ 3,6\}, \{5,7 \}, \{8,10 \}\} \in \Phi_\chi^{-1} ( \ee )$ 
is parametrized}

{\em  by the choices:
$\xlambda_4 =2$, $\xlambda_6 = 2$,
$\xlambda_7 = 1$, $\xlambda_9 =2$, $ \xlambda_{10} = 1$.
}
\end{center}


Next, it is straightforward to adjust to the $\chi$-framework the calculations 
shown in Sections 3.1 and 3.2, with the little nuisance that the $\chi$-version 
of the arguments in these sections must use a unified notation for the operators 
$A_k^{( \xsigma )}$ and $B_k^{( \xsigma )}$ from Notation \ref{def:43}.  The 
end result of all this is a statement very similar to Proposition \ref{prop:49}, 
as follows.
\end{remark}

\begin{proposition}    \label{prop:53}
Consider the data fixed at the beginning of the present subsection
($d,n,q, \ee, \chi$ and the vectors $u_1, \ldots , u_{2n} \in \bC^d$),
and let $M \in B( \cT_{d;q} )$ be the operator defined in Equation 
(\ref{eqn:51b}).  Then
\begin{equation}     \label{eqn:53xa}
M ( \xivac ) = \sum_{\pi \in \Phi_{\chi}^{-1} ( \ee )} 
\,  \Bigl[ \, q^{\emph{\Cr} ( \pi )} \cdot 
\prod_{  \begin{array}{c}
{\scriptstyle \{ k,h \} \ pair \ in \ \Ht_{\chi} \cdot \pi} \\ 
{\scriptstyle with \ \ee (h) = *, \ \ee (k) = 1}
\end{array} } 
\  \langle u_k, u_h \rangle \, \Bigr] \cdot \xivac .
\mbox{$\ $ \hspace{1cm} $\blacksquare$}
\end{equation}
\end{proposition}

Analogously to Remark \ref{rem:49}, we note that $\phivac (M)$ equals the 
scalar that appeared on the right-hand side of (\ref{eqn:53xa}), and moreover: if 
we focus on $\phivac (M)$ (rather than on the vector $M( \xivac )$), it is quite 
easy to extend the two-sided $q$-Wick formula to the case when $\ee$ would be 
picked in $\{ 1,* \}^{2n} \setminus \cD_{(1,*)} (2n)$, in which case we just get 
$\phivac (M) = 0$.

While not directly useful for the present paper, we mention that a non-trivial 
analysis of $M( \xivac )$ can actually be made for some choices of 
$\ee \in \{ 1,* \}^{2n} \setminus \cD_{(1,*)} (2n)$, and also for some choices 
of $\ee \in \{ 1,* \}^m$ with $m$ odd, by using ``incomplete'' pair-partitions 
of $\{ 1, \ldots , m \}$.  This kind of analysis was made in the one-sided 
$q$-Wick case in \cite[Proposition 2.7]{BKS1997} and \cite[Section 3]{EP2003}, and 
is pursued in depth (for several constructions of Wick products) in the recent 
paper \cite{A2017}. 

$\ $

\subsection{Proof of Proposition 1.8.}
In this subsection we fix $d,n \in \bN$ and $q \in (-1,1)$,
and we verify the moment formula stated in Proposition 1.8:
\begin{equation}   \label{eqn:52a}
( \phivac \otimes \phivac ) ( X_{d;q}^n ) 
= \widetilde{P}_n (d,q),
\end{equation}
where $X_{d;q} := \bigl( \, \sum_{i=1}^d 
( L_{i;q} + L_{i;q}^{*} ) \otimes (L_{i;q} + L_{i;q}^{*} ) \, \bigr)^2
\in B( \cT_{d;q} ) \otimes B( \cT_{d;q} )$,
and the polynomial $\widetilde{P}_n$ is as introduced in 
Notation 1.7.  A discussion very similar to the one from 
Remark \ref{rem:414} shows that the Equation (\ref{eqn:intro11}) 
used to define $\widetilde{P}_n (t,u)$ can be re-written as
\begin{equation}   \label{eqn:52b}
\widetilde{P}_n (t,u) = \sum_{\pi, \sigma \in \cP_2 (2n)}
t^{|\pi \vee \sigma |} \, u^{\Cr (\pi) + \Cr ( \sigma ) } .
\end{equation}

In the calculations shown below it will be convenient to denote
$L_{i;q} + L_{i;q}^{*} =: A_{i;q}$, $1 \leq i \leq d$.  The operators 
$A_{1;q}, \ldots , A_{d;q} \in B( \cT_{d;q} )$ go under the name of 
$q$-Gaussian random variables, and one has the following formula
(cf. \cite[Proposition 2]{BS1991}) to evaluate their joint moments 
of length $2n$ with respect to the vacuum-state: for every 
$I : \{ 1, \ldots , 2n \} \to \{ 1, \ldots , d \}$, one has
\begin{equation}   \label{eqn:52c}
\phivac ( A_{I(2n);q} \cdots A_{I(1);q} ) = 
\sum_{ \pi \in \cP_2 (2n), \ \pi \leq Ker (I) }
\ q^{\Cr ( \pi )},
\end{equation}
where $\Ker (I)$ has the same meaning (partition of 
$\{ 1, \ldots , 2n \}$ into level-sets of $I$) as in the proof 
of Lemma \ref{lemma:415}, and the inequality $\pi \leq \Ker (I)$ 
is with respect to the reverse refinement order on $\cP (2n)$.
 
When on the left-hand side of (\ref{eqn:52a}) we replace 
$X_{d;q}$ from its definition and follow the underlying algebra, 
we find that
\[
( \phivac \otimes \phivac ) ( X_{d;q}^n ) 
= \sum_{I : \{1, \ldots , 2n \} \to \{ 1, \ldots , d \}}
\Bigl( \, \phivac ( A_{I(2n);q} \cdots A_{I(1);q} ) \, \Bigr)^2.
\]
In view of (\ref{eqn:52c}), this can be continued with
\[
= \sum_{ I: \{ 1, \ldots , 2n \} \to \{ 1, \ldots , d \} } 
\ \sum_{\substack{ \pi, \sigma \in \cP_2(2n) \\ \pi, \sigma \leq \text{Ker}(I)}} 
\ q^{\Cr (\pi)} \cdot q^{\Cr (\sigma)}  .
\]
The condition ``$\pi, \sigma \leq \text{Ker}(I)$'' is equivalent to 
$\pi \vee \sigma \leq \Ker (I)$, hence when we change the order 
of summation in latter double sum we come to
\begin{equation}   \label{eqn:52d}
= \sum_{\pi, \sigma \in \cP_2(2n)} 
\ q^{\Cr (\pi)+\Cr (\sigma)}  \cdot \ \mid
 \{ I: \{ 1, \ldots , 2n \} \to \{ 1, \ldots , d \} \mid 
\Ker (I) \geq \pi \vee \sigma \} \mid .
\end{equation}
It is immediately seen that the set of $I$'s which has appeared in the 
summation (\ref{eqn:52d}) has cardinality $d^{ | \pi \vee \sigma | }$.
Hence this summation is nothing but the expression recorded in (\ref{eqn:52b}) 
for the value of $\widetilde{P}_n (d,q)$, and this concludes the verification 
of the required formula (\ref{eqn:52a}).

$\ $

\subsection{Some remarks in the case \boldmath{$q=0$}}

$\ $

\noindent
In this subsection we fix $d \in \bN$, we set $q=0$, and we make some remarks
around Proposition 1.1 (which is the special case $q=0$ of our main result, 
Theorem 1.4).  We will continue to use the framework built in Section 3, but, 
consistently with the statement of Proposition 1.1, we will omit the explicit 
occurrence of $q=0$ in notations for operators -- hence we will write 
``$L_i, R_i \in B ( \cT_d )$'' instead of 
``$L_{i;0}, R_{i_0} \in B ( \cT_{d;0} )$'', and such.  It will be moreover 
convenient to denote
\begin{equation}   \label{eqn:53a}
A_i := L_i + L_i^{*} , 
\ \ B_i := R_i + R_i^{*} , \ \ 1 \leq i \leq d.
\end{equation}
Proposition 1.1 is thus concerned with the distribution of the operator 
\begin{equation}   \label{eqn:53c}
T_d :=  A_1 B_1 + \cdots + A_d B_d \in B( \cT_d ).
\end{equation}

\begin{remark}   \label{rem:53a}
The $(2d)$-tuple of operators $A_1, \ldots , A_d, B_1, \ldots , B_d$ is
the prototypical example of bi-free Gaussian system appearing in the 
bi-free central limit theorem from \cite{V2014}.  It can also be treated 
as a ``$(2d)$-tuple of canonical operators'' in the sense of \cite{MN2015}   
(for which purpose one puts them in the form 
$A_i = L_i^{*} ( I + L_1^2 + \cdots + L_d^2)$ and 
$B_i = R_i^{*} ( I + R_1^2 + \cdots + R_d^2)$, 
$1 \leq i \leq d$).  As a consequence, one has an explicit summation formula 
which describes the joint moments of this $(2d)$-tuple, and which can be used 
in order to give a short proof of Proposition 1.1.  The way to think of this 
summation formula is as a special case of ``moment-cumulant formula for bi-free 
cumulants''.  For a detailed presentation of the bi-free cumulant machinery we 
refer the reader to \cite{CNS2015, CNS2015b}, here we only state the formula 
relevant for Proposition 1.1.  

First, a bit of notation: for every $n \in \bN$ let us put
\begin{equation}   \label{eqn:53e}
\BNC_2^{( \alt )} (2n) := \{ \Ht_n \cdot \pi \mid \pi \in \NC_2 (2n) \} ,
\end{equation}
where $\Ht_n$ is the permutation of $\{ 1, \ldots , 2n \}$ introduced in 
Notation \ref{def:33}.  The pair-partitions in $BNC_2^{( \alt )} (2n)$ are precisely 
those which are said to be {\em bi-non-crossing} with respect to the alternating 
$(2n)$-tuple $( \ell, r, \ell, r, \ldots , \ell , r) \in \{ \ell , r \}^{2n}$.
They appear as indexing set in the following formula for alternating joint moments 
of $A_1, \ldots , A_d, B_1, \ldots , B_d$: for every $n \in \bN$ and every 
$(2n)$-tuple $I: \{ 1, \ldots , 2n \} \to \{ 1, \ldots , d \}$ one has
\begin{equation}   \label{eqn:53g}
\phivac \Bigl( A_{I(2n)} B_{I(2n-1)} \cdots A_{I(2)} B_{I(1)} \Bigr) =
\sum_{\pi \in \BNC_2^{( \alt )} (2n)} \ \term (I, \pi ),
\end{equation}
where for $\pi \in \BNC_2^{( \alt )} (2n)$ we put
\[
\term (I, \pi ) := \left\{  \begin{array}{ll}
1, &  \mbox{ if $\pi \leq \Ker (I)$}   \\
0, &  \mbox{ otherwise.}
\end{array}  \right.
\]
\end{remark}

\begin{outline}
{\em (Outline of proof of Proposition 1.1 via Equation (\ref{eqn:53g}).)}

\noindent
We fix an $n \in \bN$ and we will verify that $\phivac ( T_d^n ) = Q_n (d)$, 
where $Q_n$ is the $n$-th semi-meander polynomial from Equation (\ref{eqn:intro5}). 
To this end we start in the same way as in the beginning of the proof of 
Theorem 1.4, and we derive a formula analogous to (\ref{eqn:413yb}) from that
proof, but where we now keep the $A_i$'s and $B_i$'s as they are, rather than 
breaking them in terms of $L_i, L_i^{*}, R_i, R_i^{*}$.  We get 
\begin{equation}   \label{eqn:53f}
\phivac ( T_d^n ) = 
\sum_{\substack{ I : \{ 1, \ldots , 2n \} \to \{ 1, \ldots , d \}      \\ 
                   \text{with}\,I(2k-1) = I(2k), \, 1 \leq k \leq n } } 
\phivac \Bigl( A_{I(2n)} B_{I(2n-1)} \cdots A_{I(2)} B_{I(1)} \Bigr) .
\end{equation}
We next plug Equation (\ref{eqn:53g}) into (\ref{eqn:53f}), and change the order 
of summation in the ensuing double sum.  By working a bit the inside summation 
over $I$ (in a way similar to what we did in Section 3.4 for the proof of 
Theorem 1.4) we come to 
\begin{equation}   \label{eqn:53h}
\phivac ( T_d^n ) = \sum_{\pi \in \BNC_2^{( \alt )} (2n)}
\, d^{| \pi \vee \{ \{ 1,2 \}, \ldots , \{ 2n-1 , 2n \} \} |}.
\end{equation}

Finally, in the summation on the right-hand side of (\ref{eqn:53h}) we perform 
the substitution $\pi = \Ht_n \cdot \widecheck{\pi}$, with 
$\widecheck{\pi}$ running in the set $NC_2 (2n)$ of usual non-crossing
pair-partitions.  In connection to this substitution we also recall (cf. Equation 
(\ref{eqn:33b}) in Notation \ref{def:33}) that we can replace 
$\{ \{ 1,2 \}, \ldots , \{ 2n-1 , 2n \} \}$ as $ \Ht_n \cdot \rho_{2n}$,
where $\rho_{2n}$ is the rainbow pair-partition from Equation (\ref{eqn:intro4}).
The exponent in the general term of the sum on the right-hand side of 
(\ref{eqn:53h}) thus becomes 
\[
| \pi \vee \{ \{ 1,2 \}, \ldots , \{ 2n-1 , 2n \} \} |
= | ( \Ht_n \cdot \widecheck{\pi} ) \vee ( \Ht_n \cdot \rho_{2n} ) |
\]
\[
= | ( \Ht_n \cdot ( \widecheck{\pi} \vee \rho_{2n} ) |
= | \widecheck{\pi} \vee \rho_{2n} |.
\]
Hence the sum in (\ref{eqn:53h}) is turned by the substitution into 
$\sum_{\widecheck{\pi} \in \NC_2 (2n)} d^{| \widecheck{\pi} \vee \rho_{2n} |}$,
which is precisely the formula for $Q_n (d)$.
\hfill  $\blacksquare$
\end{outline}

\begin{remark}   \label{rem:53b} 
{\em (Approach via random matrix model.)}
Yet another proof of Proposition 1.1 can be obtained by taking 
advantage of a random matrix model for semi-meander polynomials 
that was proposed in Section 5.4 of \cite{DFGG1997}.  We outline 
here how this goes.  For every $N \in \bN$, let 
$G_1^{(N)}, \ldots , G_d^{(N)}$ be a $d$-tuple of independent Gaussian 
Hermitian random matrices (as reviewed for instance on pages 368-371 
of Lecture 22 in \cite{NS2006}), and let $\tr_N$ denote the normalized 
trace on complex $N \times N$ matrices.  Equation (5.20) in Section 5.4 
of \cite{DFGG1997} says that for every $n \in \bN$ (and for the 
$d \in \bN$ which is fixed in the present subsection) one has
\begin{equation}   \label{eqn:53bx}
Q_n (d) = \lim_{N \to \infty} \Bigl[ 
\, \sum_{J: \{ 1, \ldots , n \} \to \{ 1, \ldots , d \}} 
\ ( \bE \circ \tr_N )
\bigl( \, G_{J(1)}^{(N)} \cdots G_{J(n)}^{(N)}
G_{J(n)}^{(N)} \cdots G_{J(1)}^{(N)} \, \bigr) \, \Bigl] .
\end{equation}
Note that the the number of terms in the sum on the right-hand side 
of (\ref{eqn:53bx}) is $d^n$, independent of $N$.  It is natural to ask if 
it isn't the case that each of these $d^n$ terms, taken separately, has its 
own limit for $N \to \infty$.  The considerations from Section 4 of the paper 
\cite{S2017} on bi-matrix models (specifically, Theorems 4.10 and 4.13 there) 
assure us that the separate limits of terms do indeed exist; and more precisely, 
for any fixed $J : \{ 1, \ldots , n \} \to \{ 1, \ldots , d \}$ one has
\begin{equation}   \label{eqn:53by}
\begin{split}
\lim_{N \to \infty} 
\ ( \bE \circ \tr_N )
\bigl( \, G_{J(1)}^{(N)} \cdots &G_{J(n)}^{(N)}
G_{J(n)}^{(N)} \cdots G_{J(1)}^{(N)} \, \bigr) \\
&= \phivac ( A_{J(1)} \cdots A_{J(n)} B_{J(1)} \cdots B_{J(n)} ),
\end{split}
\end{equation}
where $A_1, \ldots , A_d, B_1, \ldots , B_d$ are as in Equation 
(\ref{eqn:53a}).  Since $A_i B_j = B_j A_i$ for all $1 \leq i,j \leq d$, the 
right-hand side of Equation (\ref{eqn:53by}) can also be written as 
\begin{equation}   \label{eqn:53bz}
\phivac ( A_{J(1)} B_{J(1)} \cdots A_{J(n)} B_{J(n)} ).
\end{equation}
Finally, summing in (\ref{eqn:53bz}) over all tuples 
$J : \{ 1, \ldots , n \} \to \{ 1, \ldots , d \}$ leads 
to $\phivac \bigl( \, (A_1 B_1 + \cdots + A_d B_d)^n \, \bigr)$;
hence Equation (\ref{eqn:53bx}) implies the required 
formula (\ref{eqn:intro7}) of Proposition 1.1.
\end{remark}

\begin{remark}   \label{rem:53c} 
It would be interesting to know what is the spectrum of the 
operator $T_d$, in particular if the spectrum is an interval.
The spectral radius ( = norm) of $T_d$ goes in a regime of 
``constant times $d$'' (since it is clear that 
$|| T_d || \leq 4d$, while on the other hand 
$|| T_d || \geq || T_d ( \xivac ) || = \sqrt{d + d^2}$).
We note that, despite being a sum of $d$ operators with 
standard free Poisson distribution (where the said distribution 
is supported on the interval $[0,4]$), the operator $T_d$ is 
not positive, for instance
$\langle T_d (e_1\otimes e_2-e_2\otimes e_1), 
e_1\otimes e_2-e_2\otimes e_1\rangle < 0$. 
\end{remark}

$\ $
 
{\bf Acknowledgement.}  A.N. would like to thank Octavio Arizmendi
and Kamil Szpojankowski for useful discussions around the operator 
$T_d$, during the summer of 2016.

\end{document}